\documentclass[10pt,english]{smfart}
\usepackage[T1]{fontenc}
\usepackage[english,french]{babel}
\usepackage{latexsym,amscd,color}
\usepackage{amsmath,amsfonts,amssymb,mathrsfs}
\usepackage{enumerate,euscript}
\usepackage{amssymb,url,xspace,smfthm}
\usepackage{hyperref}
\input xy
\xyoption{all}

\DeclareMathOperator{\gal}{Gal}

\newcommand{\BibTeX}{{\scshape Bib}\kern-.08em\TeX}

\newcommand{\T}{\S\kern .15em\relax }
\newcommand{\AMS}{$\mathcal{A}$\kern-.1667em\lower.5ex\hbox
        {$\mathcal{M}$}\kern-.125em$\mathcal{S}$}

\DeclareMathOperator{\im}{Im}

\DeclareMathOperator{\spm}{Spm}

\DeclareMathOperator{\pr}{pr}
\DeclareMathOperator{\rg}{rk}

\DeclareMathOperator{\spec}{Spec}

\renewcommand{\P}{\mathbb{P}}
\newcommand{\wmu}{\widehat{\mu}}
\newcommand{\C}{\mathbb{C}}
\newcommand{\Q}{\mathbb{Q}}

\newcommand{\adeg}{\widehat{\deg}}

\newcommand{\p}{\mathfrak{p}}

\newcommand{\E}{\overline{E}}
\newcommand{\F}{\overline{F}}
\newcommand{\sE}{\mathcal{E}}
\newcommand{\G}{\overline{G}}

\renewcommand{\O}{\mathcal{O}}
\renewcommand{\L}{\overline{L}}
\newcommand{\sm}{\mathfrak{m}}

\newcommand{\f}{\mathbb{F}}
\newcommand{\ndot}{\raisebox{.4ex}{.}}

\title{On the global determinant method}
\alttitle{Autour de la m\'ethode globale de d\'eterminant}

\date{\today}
\author{Chunhui Liu}
\address{Institute for Advanced Study in Mathematics\\
Harbin Institute of Technology\\
150001 Harbin\\P. R. China}
\email{chunhui.liu@hit.edu.cn}

\begin{document}
\def\smfbyname{}
\begin{abstract}
In this paper, we build the global determinant method of Salberger by Arakelov geometry explicitly. As an application, we study the dependence on the degree of the number of rational points of bounded height in plane curves. We will also explain why some constants will be more explicit if we admit the Generalized Riemann Hypothesis.
\end{abstract}
\begin{altabstract}
Dans cet article, on construit la m\'ethode globale de d\'eterminant de Salberger par la g\'eom\'etrie d'Arakelov explicitement. Comme une application, on \'etudie la d\'ependance du degr\'e du nombre de points rationnels de hauteur major\'ee dans courbes planes. On expliquera aussi pourquoi certaines constantes seront plus explicites si on admet l'hypoth\`ese g\'en\'eralis\'ee de Riemann.
\end{altabstract}

\maketitle

\tableofcontents
\section{Introduction}
Let $X\hookrightarrow\mathbb P^n_K$ be a projective variety over a number field $K$. For every rational point $\xi\in X(K)$, we denote by $H_K(\xi)$ the height (see \eqref{classic absolute height} for the definition) of $\xi$ with respect to the above closed immersion, for example, the classic Weil height (cf. \cite[\S B.2, Definition]{Hindry}). Let \[S(X;B)=\{\xi\in X(K)\mid H_K(\xi)\leqslant B\},\] where $B\geqslant1$ and the embedding morphism is omitted. By the Northcott's property, the cardinality $\#S(X;B)$ is finite for a fixed $B\in\mathbb R$.

In order to understand the density of the rational points of $X$, it is an important approach to study the function $\#S(X;B)$ with the variable $B\in\mathbb R^+$. For different required properties of $\#S(X;B)$, numerous methods have been applied. In this article, we are interested in the uniform upper bound of $\#S(X;B)$ for all $X\hookrightarrow\mathbb P^n_K$ with fixed degree and dimension, and for those satisfying certain common conditions.
\subsection{Determinant mathod}
In order to understand the function $\#S(X;B)$ of the variable $B\in\mathbb R^+$, we will introduce the so-called \textit{determinant method} to study the number of rational points with bounded height in arithmetic varieties, which was proposed in \cite{Heath-Brown}.
\subsubsection{Basic ideas and history}
Traditionally, the determinant method is proposed over the rational number field $\Q$ to avoid some extra technical troubles. In \cite{Bombieri_Pila} (see also \cite{Pila95}), Bombieri and Pila proposed a method of determinant argument to study plane affine curves. The monomials of a certain degree evaluated on a family of rational points in $S(X;B)$ having the same reduction modulo some prime numbers form a matrix whose determinant is zero by a local estimate. By this method, they proved $\#S(X;B)\ll_{\delta,\epsilon}B^{2 /\delta+\epsilon}$ for all $\epsilon>0$, where $\delta=\deg(X)$.

In \cite{Heath-Brown}, Heath-Brown generalized the method of \cite{Bombieri_Pila} to the higher dimensional case. His idea is to focus on a subset of $S(X;B)$ whose reductions modulo a prime number are a same regular point, and he proved that this subset can be covered by a bounded degree hypersurface which do not contain the generic point of $X$. Then he counted the number of regular points over finite fields, and controled the regular reductions. In \cite{Broberg04}, Broberg generalized it to the case over arbitrary number fields.

In \cite{Serre1997,Serre2008}, Serre asked whether $\#S(X;B)\ll_{X}B^{\dim(X)}(\log B)^c$ is verified for all arithmetic varieties $X$ with a particular constant $c$. In \cite{Heath-Brown}, Heath-Brown proposed a uniform version $\#S(X;B)\ll_{d,\delta,\epsilon}B^{d+\epsilon}$ for all $\epsilon>0$ with $\delta=\deg(X)$ and $d=\dim(X)$, which is called \textit{the dimension growth conjecture}. He proved this conjecture for some special cases. Later, Browning, Heath-Brown and Salberger had some contributions on this subject, see \cite{Browning_Heath06I,Browning_Heath06II,Bro_HeathB_Salb} for the improvements of the determinant method and the proofs under certain conditions. In \cite{Salberger07}, Salberger considered the general reductions, and the multiplicities of rational points were taken into consideration, and he proved the dimension growth conjecture with certain conditions on the subvarieties of $X$.

\subsubsection{A global version}
The so-called \textit{global determinant method} was first introduced by Salberger in \cite{Salberger_preprint2013} in order to study the dimension growth conjecture mentioned above. In general, it allows one to use only one auxiliary hypersurface to cover the rational points of bounded height, and one needs to optimize the degree of this hypersurface. By the global version, he proved the dimension growth conjecture for $\deg(X)=\delta\geqslant4$ and $\#S(X;B)\ll_{\delta}B^{\frac{2}{\delta}}\log B$ when $X$ is a curve.

In \cite{Walsh_2015}, Walsh refined the global determinant method in \cite{Salberger_preprint2013}, and he removed the $\log B$ term in \cite{Salberger_preprint2013} when $X$ is a curve.

\subsubsection{The dependence on degree}
Let $X\hookrightarrow\mathbb P^n_{\mathbb Q}$ be an geometrically integral variety of degree $\delta$ and dimension $d$. We are also interested in the dependence of the uniform upper bound of $\#S(X;B)$ on $\delta$, in particular when $X$ is a plane curve ($n=2$ and $d=1$).

In \cite{Walkowiak2005}, Walkowiak studied this problem by counting integral points over $\mathbb Z$. In \cite[Th\'eor\`eme 2.10]{Motte2019}, Motte obtained an estimate, which has a better dependence on $B$ but a worse dependence on $\delta$ than that in \cite[Th\'eor\`eme 1]{Walkowiak2005}.

In fact, one is able to obtain a better dependence on $\delta$ by the global determinant method. In \cite{Cluckers2019}, Castryck, Cluckers, Dittmann and Nguyen improved \cite{Walsh_2015} on giving an explicit dependence on $\delta$. As applications, they obtained $\#S(X;B)\ll\delta^4B^{2/\delta}$ when $X$ is a plane curve, and a better partial result of the dimension growth conjecture than that in \cite{Salberger_preprint2013}, and the estimates in \cite{Walkowiak2005} and \cite{Motte2019} for the case of plane curves.

In \cite{Paredes-Sasyk2021}, the work \cite{Cluckers2019} was generalized over an arbitrary global field. Before \cite{Paredes-Sasyk2021} was announced, Vermeulen studied the case over $\mathbb F_q(t)$ in \cite{Vermeulen2020}.

Besides the uniform bounds of rational points,  \cite{Cluckers2019} also studied the bound of $2$-torsion points of the class group of number fields, which improved the work \cite[Theorem 1.1]{2-torsion} of Bhargava, Shankar, Taniguchi, Thorne, Tsimerman and Y. Zhao.


\subsubsection{Formulation by Arakelov geometry}
In \cite{Chen1,Chen2}, H. Chen reformulated the works of Salberger \cite{Salberger07} by Bost's slope method from Arakelov geometry developed in \cite{BostBour96}. In this formulation, H. Chen replaced the matrix of monomials by the evaluation map which sends a global section of a particular line bundle to its values on a family of rational points. By the slope inequalities, we can control the height of the evaluation map in the slope method, which replaces the role of Siegel's lemma in controlling heights.

There are two advantages by the approach of Arakelov geometry. First, Arakelov geometry gives a natural conceptual framework for the determinant method over arbitrary number fields. Next, it is easier to obtain explicit estimates, since the constants obtained from the slope inequalities are given explicitly in general. I has a apple.

\subsection{A global version with the formulation of Arakelov geometry}
In this article, we will construct the global determinant method over an arbitrary number field by Arakelov geometry following the strategy of \cite{Chen1,Chen2}. As a direct application, we will study the problem of counting rational points in plane curves, and we consider how these upper bounds depend on the degree. Some ideas are inspired by \cite{Salberger_preprint2013,Walsh_2015,Cluckers2019}.
\subsubsection{Main results}
First, we have the control of auxiliary hypersurface below in Corollary \ref{uniform upper bound of the degree of auxiliary hypersurface} and Corollary \ref{uniform upper bound of the degree of auxiliary hypersurface2}, which are deduced from Theorem \ref{upper bound of degree of general case}.
\begin{theo}
  Let $X$ be a geometrically integral hypersurface in $\mathbb P^n_K$ of degree $\delta$, and \[a_n = \dfrac{n}{{(n-1)\delta^{1/(n-1)}}}\] be a constant depending on $n$. Then there is a hypersurface of degree $\varpi$ which covers $S(X;B)$ but does not contain the generic point of $X$. In addition, we have
  \[\varpi\ll_{K,n}\delta^3B^{a_n}\]
  in Corollary \ref{uniform upper bound of the degree of auxiliary hypersurface}, and
  \[\varpi\ll_{K,n}\delta^{3-{1}/{(n-1)}}B^{a_n}\max\left\{\frac{\log B}{[K:\Q]},1\right\}\]
  in Corollary \ref{uniform upper bound of the degree of auxiliary hypersurface2}.

  If we assume the Generalized Riemann Hypothesis, the above constants depending on $K$ and $n$ will be given explicitly in Corollary \ref{uniform upper bound of the degree of auxiliary hypersurface} and Corollary \ref{uniform upper bound of the degree of auxiliary hypersurface2}.
\end{theo}
Since we apply the approach of Arakelov geometry in this article, we do not use the technique of "change of coordinate" in \cite[\S 3]{Walsh_2015} and \cite[\S 3.4]{Cluckers2019} any longer. Instead, we are able to obtain a uniform estimate directly.
\subsubsection{Potential applications}
Similar to the previous applications of the determinant, the above estimates are able to be applied to study the uniform upper bound of the number of rational points with bounded height, where we will initiate the induction on the dimension as usual and the study of the distribution of the loci of small degree in a variety (see \cite[\S 4]{Salberger_preprint2013} for such an example, which considered the density of conics in a cubic surface). By the above operation, we are able to obtain estimates of general arithmetic varieties from those of hypersurfaces via a suitable linear projection.

Since our method works over arbitrary number field and gives an explicit estimate (or under some technical conditions), the further applications is possible to go well under the same conditions and also be explicit.

As a direct application, we have the following results on counting rational points of bounded height in plane curves in Theorem \ref{upper bound of rational points of bounded height in plane curves} and Theorem \ref{upper bound of rational points of bounded height in plane curves2}.
\begin{theo}
  Let $X$ be a geometrically integral plane curve in $\mathbb P^2_K$ of degree $\delta$. Then we have
  \[\#S(X;B)\ll_{K}\delta^4 B^{{2}/{\delta}}\]
  in Theorem \ref{upper bound of rational points of bounded height in plane curves}, and
  \[\#S(X;B)\ll_K\delta^3B^{{2}/{\delta}}\log B\]
  in Theorem \ref{upper bound of rational points of bounded height in plane curves2}.

If we assume the Generalized Riemann Hypothesis, the above constants depending on $K$ will be given explicitly in Theorem \ref{upper bound of rational points of bounded height in plane curves} and Theorem \ref{upper bound of rational points of bounded height in plane curves2}.
\end{theo}
Theorem \ref{upper bound of rational points of bounded height in plane curves} generalizes \cite[Theorem 2]{Cluckers2019} over an arbitrary number field, and gives an explicit estimate under the assumption of the Generalized Riemann Hypothesis. Theorem \ref{upper bound of rational points of bounded height in plane curves2} can be viewed as a projective analogue of \cite[Theorem 3]{Cluckers2019} over an arbitrary number field, and a better partial result of the conjecture of Heath-Brown referred at Remark \ref{conj of HB}. These two estimates are better than those given in \cite[Th\'eor\`eme 1]{Walkowiak2005} and \cite[Th\'eor\`eme 2.10]{Motte2019}.
\subsubsection{The role of the Generalized Riemann Hypothesis}
In this work, some explicit estimates of the distribution of primes ideals are applied. If we admit \textbf{GRH (the Generalized Riemann Hypothesis)} of the Dedekind zeta function of the base number field, we are able to obtain more explicit estimates, see \cite{grenie2016explicit}, for example. Without the assumptions of GRH, it seems to be very difficult to obtain such explicit estimates over an arbitrary number field, since we do not know the zero-free region of the Dedekind zeta function. If we know the zero-free region clearly enough, for example, if we work on the rational number field $\Q$ or totally imaginary fields (see \cite{Ten} and \cite{grzeskowiak2017explicit} respectively), or we just want an implicit estimate (see \cite{Rosen1999}), we do not need to suppose GRH.
\subsection{Organization of article}
This paper is organized as following. In \S 2, we provide some preliminaries to construct the determinant method. In \S 3, we formulate the global determinant method by the slope method. In \S 4, we give some useful estimates on the non-geometrically integral reductions, a count of multiplicities over finite fields, the distributions of some particular prime ideals, and the geometric Hilbert-Samuel function. In \S 5, we provide an explicit upper bound of the determinant and lower bounds of auxiliary hypersurfaces. In \S 6, we give two uniform upper bounds of rational points of bounded height in plane curves. In \S 7, under the assumption of GRH, we give some explicit estimates of the distribution of prime ideals of bounded norm in a ring of integers, and explain how to apply these explicit estimates in the global determinant to get more explicit estimates. In Appendix A, we will give an explicit lower bound of a useful function induced by the local Hilbert-Samuel function.
\subsection*{Acknowledgement}
I would like to thank Prof. Per Salberger for introducing his brilliant work \cite{Salberger_preprint2013} to me, and for explaining to me some ingredients of his work. These discussions and suggestions play a significant role in this paper. I would also like to thank Prof. Stanley Yao Xiao for his suggestions on the study of the distribution of prime ideals. In particular, I would like to thank the anonymous referee for all the suggestions on revising the manuscript of the paper. Chunhui Liu was supported by Fundamental Research Funds for the Central Universities FRFCU5710010421.
\section{Fundamental settings}
In this section, we will introduce some preliminaries to understand the problem of counting rational points of bounded height. In particular, we will provide some basic notions in Arakelov geometry.
\subsection{Counting rational points of bounded height}
Let $K$ be a number field, and $\O_K$ be its ring of integers. We denote by $M_{K,f}$ the set of finite places of $K$, and by $M_{K,\infty}$ the set of infinite places of $K$. In addition, we denote by $M_K=M_{K,f}\sqcup M_{K,\infty}$ the set of places of $K$. For every $v\in M_K$ and $x\in K$, we define the absolute value $|x|_v=\left|N_{K_v/\Q_v}(x)\right|_v^\frac{1}{[K_v:\Q_v]}$ for each $v\in M_K$, extending the usual absolute values on $\Q_p$ or $\mathbb{R}$. Here $\mathbb Q_v$ denotes the $p$-adic field $\mathbb Q_p$, where $v$ is extended from $p$ under the extension $K/\Q$.

 Let $\xi=[\xi_0:\cdots:\xi_n]\in\mathbb P^n_K(K)$. We define the \textit{height} of $\xi$ in $\mathbb P^n_K$ as
\begin{equation}\label{classic absolute height}
  H_K(\xi)=\prod_{v\in M_K}\max_{0\leqslant i\leqslant n}\left\{|\xi_i|_v^{[K_v:\Q_v]}\right\}.
\end{equation}
We also define the logarithmic height of $\xi$ as
\begin{equation}\label{log height}
  h(\xi)=\frac{1}{[K:\Q]}\log H_K(\xi),
\end{equation}
which is invariant under the extensions over $K$ (cf. \cite[Lemma B.2.1]{Hindry}).

Suppose that $X$ is a closed integral subscheme of $\mathbb P^n_K$, and $\phi:X\hookrightarrow\mathbb P^n_K$ is the closed immersion. For $\xi\in X(K)$, we define $H_K(\xi)=H_K(\phi(\xi))$, and usually we omit the closed immersion $\phi$ if there is no confusion. Next, we denote
\[S(X;B)=\{\xi\in X(K)|H_K(\xi)\leqslant B\},\hbox{ and } N(X;B)=\#S(X;B).\]
By the Northcott's property (cf. \cite[Theorem B.2.3]{Hindry}), the cardinality $N(X;B)$ is finite for a fixed $B\geqslant1$.

\subsection{A function induced by local Hilbert-Samuel functions}
In this part, we will introduce a function induced by the local Hilbert-Samuel function of schemes at a closed point, and we will use this function in Proposition \ref{local control by Q_r}. For the motivation and background, see \cite[\S 2]{Salberger07} and \cite[\S 3.2]{Chen2}.

Let $k$ be a field, and $X$ be a closed subscheme of $\mathbb P^n_k$ of pure dimension $d$, which means all its irreducible components have the same dimension. Let $\xi$ be a closed point of $X$. We denote by
\begin{equation}\label{local hilbert of a closed point}
  H_\xi(s)=\dim_{\kappa(\xi)}\left(\sm_{X,\xi}^s/\sm_{X,\xi}^{s+1}\right)
\end{equation}
the local Hilbert-Samuel function of $X$ at the point $\xi$ with the variable $s\in\mathbb N$, where $\sm_{X,\xi}$ is the maximal ideal of the local ring $\O_{X,\xi}$, and $\kappa(\xi)$ is the residue field of the local ring $\O_{X,\xi}$. For this function, we have the polynomial asymptotic
\begin{equation}\label{definition of multiplicity}
H_\xi(s)=\frac{\mu_\xi(X)}{(d-1)!}s^{d-1}+o(s^{d-1})
\end{equation}
when $s\rightarrow+\infty$, and we define the positive integer $\mu_\xi(X)$ as the \textit{multiplicity} of point $\xi$ in $X$.

We define the series $\{q_\xi(m)\}_{m\geqslant0}$ as the increasing series of non-negative integers such that every integer $s\in\mathbb{N}$ appears exactly $H_\xi(s)$ times in this series. For example, if $H_\xi(0)=1$, $H_\xi(1)=2,H_{\xi}(2)=4,H_{\xi}(3)=5,\ldots$, then the series $\{q_\xi(m)\}_{m\geqslant0}$ is
\[\{0,1,1,2,2,2,2,3,3,3,3,3,4,\ldots\}.\]
Let $\{Q_\xi(m)\}_{m\geqslant0}$ be the partial sum of the series $\{q_\xi(m)\}_{m\geqslant0}$, which is
 \begin{equation}\label{Q_r}
   Q_\xi(m)=q_\xi(0)+q_\xi(1)+\cdots+q_\xi(m)
 \end{equation}
for all $m\in\mathbb N$.

If $X$ is a hypersurface of $\mathbb P^n_k$, then by \cite[Example 2.70 (2)]{Kollar2007}, the local Hilbert-Samuel function of $X$ at the point $\xi$ defined in \eqref{local hilbert of a closed point} is
\begin{equation*}
  H_\xi(s)={n+s-1\choose s}-{n+s-\mu_\xi(X)-1\choose s-\mu_\xi(X)}.
\end{equation*}
In this case, we have the following explicit lower bound of $Q_\xi(m)$, which is
  \begin{equation}\label{lower bound of Q_r}
    Q_\xi(m)>\left(\frac{(n-1)!}{\mu_\xi(X)}\right)^{\frac{1}{n-1}}\left(\frac{n-1}{n}\right)m^{\frac{n}{n-1}}-\frac{n^3+2n^2+n-4}{2n(n+1)}m.
  \end{equation}
  This lower bound has the optimal dominant term by the argument in \cite[Main Lemma 2.5]{Salberger07} and some other subsequent references. In Appendix \ref{Section: estimate of Q}, we will provide a detailed proof of this lower bound.

\subsection{Normed vector bundles}

   A \textit{normed vector bundle} over $\spec\O_K$ is all the pairings $\E=\left(E,\left(\|\ndot\|_v\right)_{v\in M_{K,\infty}}\right)$, where:
  \begin{itemize}
    \item $E$ is a projective $\O_K$-module of finite rank;
    \item $\left(\|\ndot\|_v\right)_{v\in M_{K,\infty}}$ is a family of norms, where $\|\ndot\|_v$ is a norm over $E\otimes_{\O_K,v}\C$ which is invariant under the action of $\gal(\C/K_v)$. We consider a complex place and its conjugation as two different places.
  \end{itemize}

If for all $v\in M_{K,\infty}$, the norms $\left(\|\ndot\|_v\right)_{v\in M_{K,\infty}}$ are Hermitian, we say that $\E$ is a \textit{Hermitian vector bundle} over $\spec\O_K$. If $\rg_{\O_K}(E)=1$, we say that $\E$ is a \textit{Hermitian line bundle}.

Suppose that $F$ is a sub-$\O_K$-module of $E$. We say that $F$ is a \textit{saturated} sub-$\O_K$-module if $E/F$ is a torsion-free $\O_K$-module.

Let $\E=\left(E,\left(\|\ndot\|_{E,v}\right)_{v\in M_{K,\infty}}\right)$ and $\F=\left(F,\left(\|\ndot\|_{F,v}\right)_{v\in M_{K,\infty}}\right)$ be two Hermitian vector bundles. If $F$ is a saturated sub-$\O_K$-module of $E$ and $\|\ndot\|_{F,v}$ is the restriction of $\|\ndot\|_{E,v}$ over $F\otimes_{\O_K,v}\C$ for every $v\in M_{K,\infty}$, we say that $\F$ is a \textit{sub-Hermitian vector bundle} of $\E$ over $\spec\O_K$.

We say that $\G=\left(G,\left(\|\ndot\|_{G,v}\right)_{v\in M_{K,\infty}}\right)$ is a \textit{quotient Hermitian vector bundle} of $\E$ over $\spec\O_K$, if for every $v\in M_{K,\infty}$, the module $G$ is a projective quotient $\O_K$-module of $E$ and $\|\ndot\|_{G,v}$ is the induced quotient space norm of $\|\ndot\|_{E,v}$.

For simplicity, we will denote by $E_K=E\otimes_{\O_K}K$ below.

\subsection{Arakelov invariants}
Let $\E$ be a Hermitian vector bundle over $\spec\O_K$, and $\{s_1,\ldots,s_r\}$ be a $K$-basis of $E_K$. We will introduce some invariants in Arakelov geometry below.
\subsubsection{Arakelov degree}
The \textit{Arakelov degree} of $\E$ is defined as
\begin{eqnarray*}
  \adeg(\E)&=&-\sum_{v\in M_{K}}[K_v:\Q_v]\log\left\|s_1\wedge\cdots\wedge s_r\right\|_v\\
  &=&\log\left(\#\left(E/\O_Ks_1+\cdots+\O_Ks_r\right)\right)-\frac{1}{2}\sum_{v\in M_{K,\infty}}[K_v:\Q_v]\log\det\left(\langle s_i,s_j\rangle_{v,1\leqslant i,j\leqslant r}\right),
\end{eqnarray*}
where $\left\|s_1\wedge\cdots\wedge s_r\right\|_v$ follows the definition in \cite[2.1.9]{Chen10b} for all $v\in M_{K,\infty}$, and $\langle s_i,s_j\rangle_{v,1\leqslant i,j\leqslant r}$ is the Gram matrix of the basis $\{s_1,\ldots,s_r\}$ with respect to $v\in M_{K,\infty}$. We refer the readers to \cite[2.4.1]{Gillet-Soule91} for a proof of the equivalence of the above two definitions. The Arakelov degree is independent of the choice of the basis $\{s_1,\ldots,s_r\}$ by the product formula (cf. \cite[Chap. III, Proposition 1.3]{Neukirch}). In addition, we define
\[\adeg_n(\E)=\frac{1}{[K:\Q]}\adeg(\E)\]
as the \textit{normalized Arakelov degree} of $\E$, which is independent of the choice of $K$.
\subsubsection{Slope}
Let $\E$ be a non-zero Hermitian vector bundle over $\spec\O_K$, and $\rg(E)$ be the rank of $E$. The \textit{slope} of $\E$ is defined as
\[\wmu(\E):=\frac{1}{\rg(E)}\adeg_n(\E).\]
In addition, we denote by $\wmu_{\max}(\E)$ the maximal value of slopes of all non-zero Hermitian subbundles, and by $\wmu_{\min}(\E)$ the minimal value of slopes of all non-zero Hermitian quotients bundles of $\E$.
\subsubsection{Height of linear maps}
Let $\E$ and $\F$ be two non-zero Hermitian vector bundles over $\spec\O_K$, and $\phi:\; E_K\rightarrow F_K$ be a non-zero homomorphism. The \textit{height} of $\phi$ is defined as \[h(\phi)=\frac{1}{[K:\Q]}\sum_{v\in M_K}\log\|\phi\|_v,\]
where $\|\phi\|_v$ is the operator norm of $\phi_v:E\otimes_KK_v\rightarrow F\otimes_KK_v$ induced by the above linear homomorphism with respect to $v\in M_K$.

We refer the readers to \cite[Appendix A]{BostBour96} for some equalities and inequalities on Arakelov degrees and corresponding heights of homomorphisms.

\subsection{Arithmetic Hilbert-Samuel function}\label{basic setting}
Let $\overline{\mathcal E}$ be a Hermitian vector bundle of rank $n+1$ over $\spec\O_K$, and $\mathbb P(\sE)$ be the projective space which represents the functor from the category of commutative $\O_K$-algebras to the category of sets mapping all $\O_K$-algebra $A$ to the set of projective quotient $A$-module of $\sE\otimes_{\O_K}A$ of rank $1$. Let $\O_{\mathbb P (\sE)}(1)$ (or $\O(1)$ if there is no confusion) be the universal bundle, and we denote by $\O_{\mathbb P (\sE)}(D)$ (or $\O(D)$) the line bundle $\O_{\mathbb P (\sE)}(1)^{\otimes D}$ for simplicity. The Hermitian metrics on $\sE$ induce by quotient of Hermitian metrics (i.e. Fubini-Study metrics) on $\O_{\mathbb P(\sE)}(1)$ which define a Hermitian line bundle $\overline{\O_{\mathbb P(\sE)}(1)}$ on $\mathbb P(\sE)$.

For every $D\in\mathbb N^+$, let
\begin{equation}\label{definition of E_D}
  E_D=H^0\left(\mathbb P(\sE),\O_{\mathbb P (\sE)}(D)\right),
\end{equation} and let $r(n,D)$ be its rank over $\O_K$. In fact, we have
\begin{equation}\label{def of r(n,D)}
  r(n,D)={n+D\choose D}.
\end{equation}
For each $v\in M_{K,\infty}$, we denote by $\|\ndot\|_{v,\sup}$ the norm over $E_{D,v}=E_D\otimes_{\O_K,v}\C$ such that
\begin{equation}\label{definition of sup norm}
  \forall\;s\in E_{D,v},\;\|s\|_{v,\sup}=\sup_{x\in\mathbb P(\sE_K)_v(\C)}\|s(x)\|_{v,\mathrm{FS}},
\end{equation}
where $\|\ndot\|_{v,\mathrm{FS}}$ is the corresponding Fubini-Study norm.

\subsubsection{Metric of John}\label{metric of john}
Next, we introduce the \textit{metric of John}, see \cite{Thompson96} for a systematic introduction of this notion. In general, for a given symmetric convex body $C$, there exists the unique ellipsoid, called \textit{ellipsoid of John}, contained in $C$ with the maximal volume.

For the $\O_K$-module $E_D$ and any place $v\in M_{K,\infty}$, we take the ellipsoid of John of its unit closed ball defined via the norm$\|\ndot\|_{v,\sup}$, and this ellipsoid induces a Hermitian norm, noted by $\|\ndot\|_{v,J}$. For every section $s\in E_{D}$, the inequality
\begin{equation}\label{john norm}
  \|s\|_{v,\sup}\leqslant\|s\|_{v,J}\leqslant\sqrt{r(n,D)}\|s\|_{v,\sup}
\end{equation}
is verified by \cite[Theorem 3.3.6]{Thompson96}.
\subsubsection{Evaluation map}
Let $X$ be an integral closed subscheme of $\mathbb{P}(\mathcal{E}_K)$, and $\mathscr{X}$ be the Zariski closure of $X$ in $\mathbb{P}(\mathcal{E})$. We denote by
\begin{equation}\label{evaluation map}
\eta_{X,D}:\;E_{D,K}=H^0\left(\mathbb{P}(\mathcal{E}_K),\O(D)\right)\rightarrow H^0\left(X,\O_{\mathbb P(\sE_K)}(1)|_X^{\otimes D}\right)
\end{equation}
the \textit{evaluation map} over $X$ induced by the closed immersion of $X$ in $\mathbb P(\sE_K)$. We denote by $F_D$ the largest saturated sub-$\O_K$-module of $H^0\left(\mathscr{X},\O_{\mathbb P(\sE)}(1)|_\mathscr{X}^{\otimes D}\right)$ such that $F_{D,K}=\im(\eta_{X,D})$. When the integer $D$ is large enough, the homomorphism $\eta_{X,D}$ is surjective, which means $F_D=H^0(\mathscr{X},\O_{\mathbb P(\sE)}(1)|_\mathscr{X}^{\otimes D})$.

The $\O_K$-module $F_D$ is equipped with the quotient metrics (from $\E_D$) such that $F_D$ is a Hermitian vector bundle over $\spec \O_K$, noted by $\F_D$ this Hermitian vector bundle.
\begin{defi}[Arithmetic Hilbert-Samuel function]\label{arithmetic hilbert function}
Let $\F_D$ be the Hermitian vector bundle over $\spec\O_K$ defined above from the map \eqref{evaluation map}. We say that the function which maps the positive integer $D$ to $\wmu(\F_D)$ is the \textit{arithmetic Hilbert-Samuel function} of $X$ with respect to the Hermitian line bundle $\overline{\O(1)}$.
\end{defi}

\subsection{Height of rational points}
In this part, we will define a height function of rational points by Arakelov geometry.

Let $\overline {\mathcal E}$ be a Hermitian vector bundle of rank $n+1$ over $\spec\O_K$, $P\in \mathbb P(\mathcal E_K)(K)$, and $\mathcal P\in\mathbb P(\mathcal E)(\O_K)$ be the Zariski closure of $P$ in $\mathbb P(\mathcal E)$. Let $\overline {\O_{\mathbb P(\mathcal E)}(1)}$ be the universal bundle equipped with the corresponding Fubini-Study metric at each $v\in M_{K,\infty}$, then $\mathcal P^*\overline {\O_{\mathbb P(\mathcal E)}(1)}$ is a Hermitian vector bundle over $\spec\O_K$. We define the \textit{height} of the rational point $P$ with respect to $\overline {\O_{\mathbb P(\mathcal E)}(1)}$ as
\begin{equation}\label{arakelov height}
  h_{\overline {\O_{\mathbb P(\mathcal E)}(1)}}(P)=\adeg_n\left(\mathcal P^*\overline {\O_{\mathbb P(\mathcal E)}(1)}\right).
\end{equation}

We keep all the above notations. We choose
\begin{equation}\label{l^2-sE}
\overline{\sE}=\left(\O_K^{\oplus(n+1)},\left(\|\ndot\|_v\right)_{v\in M_{K,\infty}}\right),
\end{equation}
 where for every $v\in M_{K,\infty}$, $\|\ndot\|_v$ is the $\ell^2$-norm mapping $(t_0,\ldots,t_n)$ to $\sqrt{|v(t_0)|^2+\cdots+|v(t_n)|^2}$. In this case, we use the notations $\mathbb P^n_K=\mathbb P(\mathcal E_K)$ and $\mathbb P^n_{\O_K}=\mathbb P(\mathcal E)$ for simplicity. We suppose that $P$ has a $K$-rational projective coordinate $[x_0:\cdots:x_n]$, then we have (cf. \cite[Proposition 9.10]{Moriwaki-book})
\begin{eqnarray*}
  h_{\overline {\O_{\mathbb P(\mathcal E)}(1)}}(P)&=&\sum\limits_{v\in M_{K,f}}\frac{[K_v:\Q_v]}{[K:\Q]}\log \left(\max\limits_{1\leqslant i\leqslant n}|x_i|_v\right)\\
  & &\;\;+\frac{1}{2}\sum\limits_{v\in M_{K,\infty}}\frac{[K_v:\Q_v]}{[K:\Q]}\log\left(\sum\limits_{j=0}^n|v(x_j)|^2\right).
\end{eqnarray*}
\begin{rema}
  We compare the logarithmic height $h(\ndot)$ defined in \eqref{log height} and the height $h_{\overline {\O_{\mathbb P(\mathcal E)}(1)}}(\ndot)$ defined in \eqref{arakelov height} by Arakelov geometry, where $\overline{\sE}$ is defined in \eqref{l^2-sE}. In fact, by an elementary calculation, the inequality
  \[\left|h(P)-h_{\overline {\O_{\mathbb P(\mathcal E)}(1)}}(P)\right|\leqslant\frac{1}{2}\log(n+1)\]
  is uniformly verified for all $P\in\mathbb P^n_K(K)$.
\end{rema}
Let $\psi:X\hookrightarrow\mathbb P^n_K$ be a projective scheme, and $P\in X(K)$. We define the height of $P$ as $h_{\overline {\O_{\mathbb P(\mathcal E)}(1)}}(\psi(P))$. We will just use the notation $h_{\overline {\O_{\mathbb P(\mathcal E)}(1)}}(P)$ or $h(P)$ if there is no confusion of the morphism $\psi$ and the Hermitian line bundle $\overline{\O_{\mathbb P(\mathcal E)}(1)}$.
\subsection{Height functions of arithmetic varieties}
In this part, we will introduce several height functions of arithmetic varieties, which evaluate their arithmetic complexities.
\subsubsection{Arakelov height}
First, we will introduce a kind of height functions of arithmetic varieties defined by the arithmetic intersection theory developped by Gillet and Soul\'e in \cite{Gillet_Soule-IHES90}, which is first introduced by Faltings in \cite[Definition 2.5]{Faltings91}, see also \cite[III.6]{Soule92}.
\begin{defi}[Arakelov height]\label{arakelov height of projective variety}
Let $K$ be a number field, $\O_K$ be its ring of integers, $\overline{\sE}$ be a Hermitian vector bundle of rank $n+1$ over $\spec\O_K$, and $\overline {\mathcal L}$ be a Hermitian line bundle over $\mathbb P(\sE)$. Let $X$ be a pure dimensional closed subscheme of $\mathbb P(\sE_K)$ of dimension $d$, and $\mathscr X$ be the Zariski closure of $X$ in $\mathbb P(\sE)$. The \textit{Arakelov height} of $X$ is defined as the arithmetic intersection number
\begin{equation*}
  \frac{1}{[K:\Q]}\adeg\left(\widehat{c}_1(\overline{\mathcal{L}})^{d+1}\cdot[\mathscr X]\right),
\end{equation*}
where $\widehat{c}_1(\overline{\mathcal{L}})$ is the arithmetic first Chern class of $\overline{\mathcal L}$ (cf. \cite[Chap. III.4, Proposition 1]{Soule92} for its definition). This height is noted by $h_{\overline{\mathcal{L}}}(X)$ or $h_{\overline{\mathcal{L}}}(\mathscr X)$.
\end{defi}
\begin{rema}\label{definition of arakelov height}
  With all the notations in Definition \ref{arithmetic hilbert function} and Definition \ref{arakelov height of projective variety}. By \cite[Th\'eor\`eme A]{Randriam06}, we have
  \[h_{\overline{\O(1)}}(X)=\lim_{D\rightarrow+\infty}\frac{\adeg_n(\F_D)}{D^{d+1}/(d+1)!}.\]

\end{rema}
\subsubsection{Heights of hypersurfaces}
Let $X$ be a hypersurface in $\mathbb P^n_K$ of degree $\delta$. By \cite[Proposition 7.6 (d), Chap. I]{GTM52}, $X$ is define by a homogeneous polynomial of degree $\delta$. We define a height function of hypersurfaces by considering its polynomial of definition.
\begin{defi}[Naive height]\label{classic height of hypersurface}
Let
\[f(T_0,\ldots,T_n)=\sum\limits_{(i_0,\ldots,i_n)\in\mathbb N^{n+1}}a_{i_0,\ldots,i_n}T_0^{i_0}\cdots T_n^{i_n}\in K[T_0,\ldots,T_n].\]
 We define the \textit{naive height} of $f(T_0,\ldots,T_n)$ as
\[H_K(f)=\prod_{v\in M_K}\max\limits_{(i_0,\ldots,i_n)\in\mathbb N^{n+1}}\left\{|a_{i_0,\ldots,i_n}|_v\right\}^{[K_v:\Q_v]},\]
and \[h(f)=\frac{1}{[K:\Q]}\log H_K(f).\] In addition, if $f(T_0,\ldots,T_n)$ is homogeneous and defines the hypersurface $X\hookrightarrow\mathbb P^n_K$,  we define the \textit{naive height} of $X$ as
\[H_K(X)=H_K(f)\hbox{ and }h(X)=h(f).\]
\end{defi}
\subsubsection{Comparison of height functions}
In order to compare $h_{\overline{\O(1)}}(\mathscr X)$ and $h(X)$ for a hypersurface $X$, we refer the following result in \cite{Liu-reduced}.
\begin{prop}\label{comparing heights}
  Let $X$ be a hypersurface in $\mathbb P^n_K$ of degree $\delta$. With all the notations above, we have
  \begin{eqnarray*}
    -\delta\left(\frac{1}{2\delta}\log\left((n+1)(\delta+1)\right)+\frac{1}{2}\mathcal H_n\right)&\leqslant& h(X)-h_{\overline{\O(1)}}(\mathscr X)\\
    &\leqslant&\delta\left(\log 2+5\log(n+1)-\frac{1}{2}\mathcal H_n\right),
  \end{eqnarray*}
  where $\mathcal H_n=1+\cdots+\frac{1}{n}$.
\end{prop}
\begin{proof}
  Since $X$ is a hypersurface, the Chow variety of $X$ is just $X$ itself. Then we have the result from \cite[Proposition 3.7]{Liu-reduced} directly after some elementary calculations.
\end{proof}
\section{Global determinant method for hypersurfaces}
In the rest part of this article, unless specially mentioned, we suppose that $X$ is an integral hypersurface in $\mathbb P^n_K$ , and $\mathscr X$ is its Zariski closure in $\mathbb P^n_{\O_K}$. In fact, $X\rightarrow\spec K$ is the generic fiber of $\mathscr X\rightarrow\spec\O_K$. When we consider the height $h(P)$ of a rational point $P\in X(K)$ embedded into $\mathbb P^n_K$, we use the definition in \eqref{arakelov height} by Arakelov geometry. Let $\p$ be a maximal ideal of $\O_K$, and we denote by $\mathscr X_\p=\mathscr X\times_{\spec\O_K}\spec\f_\p\rightarrow\f_\p$ the fiber at $\p$.

Let $r_1(n,D)$ be the rank of $F_D$ over $\O_K$, where $F_D$ is defined in \S \ref{basic setting}. For the case where $X$ is a hypersurface of degree $\delta$ in $\mathbb P^n_K$, we have
\[r_1(n,D)={n+D\choose n}-{n+D-\delta\choose n}.\]

 Our main target of this section is to prove the following result.
\begin{theo}\label{global determinant}
  We keep all the notations in \S \ref{basic setting} and this section. Let $X$ be a closed integral subscheme in $\mathbb P^n_k$, and $\mathscr X$ be its Zariski closure in $\mathbb P^n_{\O_K}$. Let $\{\p_j\}_{j\in J}$ be a finite family of maximal ideals of $\O_K$, and $\{P_i\}_{i\in I}$ be a family of rational points of $X$. For a fixed prime ideal $\p$ of $\O_K$, let $\mu_\xi(\mathscr X_\p)$ be the multiplicity of the point $\xi$ in $\mathscr X_\p$, and we denote $n(\mathscr X_\p)=\sum\limits_{\xi\in\mathscr X(\f_\p)}\mu_\xi(\mathscr X_\p)$. If the inequality
  \begin{eqnarray}
    & &\;\;\sup_{i\in I}h(P_i)<\frac{\wmu(\overline{F}_D)}{D}-\frac{\log r_1(n,D)}{2D}\\
    & &+\frac{1}{[K:\Q]}\sum_{j\in J}\left(\frac{(n-1)!^{\frac{1}{n-1}}(n-1)r_1(n,D)^\frac{1}{n-1}}{nDn(\mathscr X_{\p_j})^\frac{1}{n-1}}-\frac{n^3+2n^2+n-4}{2Dn(n+1)}\right)\log N(\p_j)\nonumber
  \end{eqnarray}
  is verified, then there exists a section $s\in E_{D,K}$, which contains $\{P_i\}_{i\in I}$ but does not contain the generic point of $X$. In other words, $\{P_i\}_{i\in I}$ can be covered by a hypersurfaces of degree $D$ which does not contain the generic point of $X$.
\end{theo}
\subsection{Auxiliary results}
We refer to some results in \cite{Chen1,Chen2}, which are used in the reformulation of the determinant method by Arakelov geometry. We will also prove a new auxiliary lemma.
\begin{prop}[\cite{Chen1}, Proposition 2.2]\label{slope of evaluation map}
  Let $\overline E$ be a Hermitian vector bundle of rank $r>0$ over $\spec\O_K$, and $\{\overline L_i\}_{i\in I}$ be a family of Hermitian line bundles over $\spec\O_K$. If \[\phi:\; E_K\rightarrow\bigoplus\limits_{i\in I}L_{i,K}\] is an injective homomorphism, then there exists a subset $I_0$ of $I$ whose cardinality is $r$ such that the following equality
\[\wmu(\E)=\frac{1}{r}\left(\sum_{i\in I_0}\wmu(\overline L_i)+h\left(\wedge^r\left(\pr_{I_0}\circ\phi\right)\right)\right)\]
is verified, where $\pr_{I_0}:\;\bigoplus\limits_{i\in I}L_{i,K}\rightarrow\bigoplus\limits_{i\in I_0}L_{i,K}$ is the canonical projection.
\end{prop}
In order to benefit the readers, we will provide the details on the construction of certain local homomorphisms, which are introduced in \cite[Lemma 2.4]{Salberger07}, see also \cite[\S 3.2]{Chen2}.

Let $X$ be an integral closed subscheme of $\mathbb P^n_K$ and $\mathscr X$ be the Zariski closure of $X$ in $\mathbb P^n_{\O_K}$. Let $\p$ be a maximal ideal of $\O_K$ and $\xi\in\mathscr X(\f_\p)$. In this case, $\O_{\mathscr X,\xi}$ is a local algebra over $\O_{K,\p}$. Let $(f_i)_{1\leqslant i\leqslant m}$ be a family of local homomorphisms of $\O_{K,\p}$-algebras from $\O_{\mathscr X,\xi}$ to $\O_{K,\p}$.

Let $E$ be a free sub-$\O_{K,\p}$-module of finite type of $\O_{\mathscr X,\xi}$ and let $f$ be the $\O_{K,\p}$-linear homomorphism
\[(f_i|_E)_{1\leqslant i\leqslant m}:E\rightarrow\O_{K,\p}^{\oplus m}.\]
 Since $f_1$ is a local homomorphism of $\O_{K,\p}$-algebras, it must be surjective. Let $\mathfrak a$ be the kernel of $f_1$, then we have $\O_{\mathscr X,\xi}/\mathfrak a\cong\O_{K,\p}$. Furthermore, since $\O_{\mathscr X,\xi}$ is a local ring and we suppose that $\sm_\xi$ is its maximal ideal, then we have $\sm_\xi\supseteq\mathfrak a$. Moreover, since $f_1$ is a local homomorphism, we have $\mathfrak a+\p\O_{\mathscr X,\xi}=\sm_\xi$. For each $j\in\mathbb N$, $\mathfrak a^j/\mathfrak a^{j+1}$ is an $\O_{\mathscr X,\xi}/\mathfrak a\cong\O_{K,\p}$-module of finite type.

In order to estimate its rank, we need the following result.
 \begin{lemm}\label{isomorphism a/a m/m}
   With all the above notations and constructions, we have
   \begin{eqnarray*}\f_\p\otimes_{\O_{K,\p}}(\mathfrak a^j/\mathfrak a^{j+1})&\cong&\left(\mathfrak a/\mathfrak a\cap \p\O_{\mathscr X,\xi}\right)^j/\left(\mathfrak a/\mathfrak a\cap\p\O_{\mathscr X,\xi}\right)^{j+1}\\
   &\cong&\left(\sm_\xi/\sm_\xi\cap\p\O_{\mathscr X,\xi}\right)^j/\left(\sm_\xi/\sm_\xi\cap\p\O_{\mathscr X,\xi}\right)^{j+1}.\end{eqnarray*}
 \end{lemm}
 \begin{proof}
   By definition, we have
   \[\f_\p\otimes_{\O_{K,\p}}(\mathfrak a^j/\mathfrak a^{j+1})\cong\left(\mathfrak a/\mathfrak a\cap \p\O_{\mathscr X,\xi}\right)^j/\left(\mathfrak a/\mathfrak a\cap\p\O_{\mathscr X,\xi}\right)^{j+1}.\]

  Next, from the facts $\mathfrak a+\p\O_{\mathscr X,\xi}=\sm_\xi$ and $\mathfrak a\subseteq\sm_\xi$, we claim that \[\mathfrak a+\sm_{\xi}\cap\p\O_{\mathscr X,\xi}=\sm_\xi\cap(\mathfrak a+\p\O_{\mathscr X,\xi})=\sm_\xi\] is verified. In fact, for every $x\in\sm_\xi=\sm_\xi\cap(\mathfrak a+\p\O_{\mathscr X,\xi})$, there exist $v\in\mathfrak a$ and $w\in \p\O_{\mathscr X,\xi}$, such that $x=v+w$. Then $w=x-v\in\sm_\xi$. Since $v,x\in\sm_\xi$, then $w\in\sm_\xi\cap\p\O_{\mathscr X,\xi}$. So we have $x\in\mathfrak a+\sm_\xi\cap\p\O_{\mathscr X,\xi}$. Conversely, since $\sm_\xi$ is the maximal ideal, then $\mathfrak a+\sm_{\xi}\cap\p\O_{\mathscr X,\xi}\subseteq\sm_\xi$. 

By the above fact, we have \[\mathfrak a/\mathfrak a\cap\p\O_{\mathscr X,\xi}\cong\mathfrak a/\mathfrak a\cap(\sm_\xi\cap\p\O_{\mathscr X,\xi})\cong(\mathfrak a+\sm_\xi\cap\p\O_{\mathscr X,\xi})/\sm_\xi\cap\p\O_{\mathscr X,\xi}\cong\sm_\xi/\sm_\xi\cap\p\O_{\mathscr X,\xi},\] which terminates the proof.
 \end{proof}
By Nakayama's lemma (cf. \cite[Theorem 2.2]{Matsumura}), we deduce that the rank of $\mathfrak a^j/\mathfrak a^{j+1}$ over $O_{K,\p}$ is equal to the rank of $\left(\sm_\xi/\sm_\xi\cap\p\O_{\mathscr X,\xi}\right)^j/\left(\sm_\xi/\sm_\xi\cap\p\O_{\mathscr X,\xi}\right)^{j+1}$ over $\f_\p$ from the isomorphism in Lemma \ref{isomorphism a/a m/m}, which is the value of the local Hilbert-Samuel function $H_\xi(j)$ defined in \eqref{local hilbert of a closed point}.

By this fact, we consider the filtration
\[\O_{\mathscr X,\xi}=\mathfrak a^0\supseteq\mathfrak a\supseteq\mathfrak a^2\supseteq\cdots\supseteq\mathfrak a^j\supseteq\mathfrak a^{j+1}\supseteq\cdots\]
of $\O_{\mathscr X,\xi}$, which induces the filtration
\begin{equation}
  \mathcal F:\;E=E\cap\mathfrak a^0\supseteq E\cap\mathfrak a\supseteq E\cap\mathfrak a^2\supseteq\cdots\supseteq E\cap\mathfrak a^j\supseteq E\cap\mathfrak a^{j+1}\supseteq\cdots
\end{equation}
of $E$ whose $j$-th sub-quotient $E\cap\mathfrak a^j/E\cap \mathfrak a^{j+1}$ is a free $\O_{K,\p}$-module of rank smaller than or equal to $H_\xi(j)$.

We suppose that the reductions of all the above local homomorphisms $f_1,\ldots,f_m$ modulo $\p$ are same, which means all the composed homomorphisms \[\O_{\mathscr X,\xi}\stackrel{f_i}{\longrightarrow}\O_{K,\p}\rightarrow\f_\p\] are the same for every $i =1,\ldots,m$, where the last arrow is the canonical reduction
morphism modulo $\p$. We note $N(\p) = \#\f_\p$. In this case, the restriction of $f$ on $E\cap\mathfrak a^j$ has its norm smaller than $N(\p)^{-j}$. In fact, for any $1\leqslant i\leqslant m$, we have $f_i(\mathfrak a)\subseteq \p\O_{K,\p}$ and hence we have $f_i(\mathfrak a^j)\subseteq \p^j\O_{K,\p}$.

By the above argument, we have the following result from \cite[Lemma 3.2, Lemma 3.3]{Chen2}.

\begin{prop}[\cite{Chen2}, Proposition 3.4]\label{local control by Q_r}
  Let $\p$ be a maximal ideal of $\O_K$ and $\xi\in\mathscr X(\f_\p)$. Suppose that $\{f_i\}_{1\leqslant i\leqslant m}$ is a family of local $\O_{K,\p}$-linear homomorphism from $\O_{\mathscr X,\xi}$ to $\O_{K,\p}$ whose reduction modulo $\p$ are the same. Let $E$ be a free sub-$\O_{K,\p}$-module of finite type of $\O_{\mathscr X,\xi}$ and $f=(f_i|_{E})_{1\leqslant i\leqslant m}$. Then for any integer $r\geqslant1$, we have
  \begin{equation}
    \|\wedge^rf_K\|\leqslant N(\p)^{-Q_\xi(r)},
  \end{equation}
  where $N(\p)=\#(\O_K/\p)$, and $Q_\xi(r)$ is defined in \eqref{Q_r}.
\end{prop}
The following lemma will be used in the global determinant estimate.
\begin{lemm}\label{product of determinants}
  Let $(K,|\ndot|)$ be a normed field, $E_1,E_2,F_1,F_2$ be four normed vector spaces over $K$, $f_1:E_1\rightarrow F_1$ and $f_2:E_2\rightarrow F_2$ be two $K$-linear isomorphisms. Suppose $\dim_K(E_1)=\dim_K(F_1)=r_1$ and $\dim_K(E_2)=\dim_K(F_2)=r_2$. We equipped \[f_1\oplus f_2: E_1\oplus E_2\rightarrow F_2\oplus F_2\] with the corresponding maximal value norms. Then we have
  \[\left\|\wedge^{r_1+r_2}\left(f_1\oplus f_2\right)\right\|=\|\wedge^{r_1}f_1\|\cdot\|\wedge^{r_2}f_2\|,\]
  where the above $\|\ndot\|$ is the norm of operators.
\end{lemm}
\begin{proof}
  By definition, the linear maps $\wedge^{r_1}f_1$ and $\wedge^{r_2}f_2$ are both scalar products by the corresponding determinants, and $\wedge^{r_1+r_2}\left(f_1\oplus f_2\right)$ is the scalar product of the above two determinants. Then we have the result by definition directly.
\end{proof}
\subsection{Proof of Theorem \ref{global determinant}}
We are going to prove Theorem \ref{global determinant}, and some ideas of the proof below are inspired from \cite[\S 3]{Chen2}.
\begin{proof}[Proof of Theorem \ref{global determinant}]
  Let $D$ be an integer larger than $1$. We suppose that the global section predicted by Theorem \ref{global determinant} does not exist. Then the evaluation map
  \begin{equation*}
  f:\;F_{D,K}\rightarrow\bigoplus_{i\in I} P_i^*\O_{\mathbb P^n_K}(D)
  \end{equation*}
  is injective. We can replace $I$ by one of its subsets such that $f$ is an isomorphism. From now on, we suppose $f$ is isomorphic, which means $\#I=r_1(n,D)=\rg F_{D,K}$. Then by Proposition \ref{slope of evaluation map}, we have
  \[\wmu(\overline{F}_D)=\frac{1}{r_1(n,D)}\left(D\sum_{i\in I}h(P_i)+h\left(\wedge^{r_1(n,D)}f\right)\right),\]
  which implies
    \[\frac{\wmu(\overline{F}_D)}{D}\leqslant\sup_{i\in I}h(P_i)+\frac{1}{Dr_1(n,D)}h\left(\wedge^{r_1(n,D)}f\right).\]

  Now we estimate the height of $\wedge^{r_1(n,D)}f$. For every $v\in M_{K,\infty}$, we have
  \[\frac{1}{r_1(n,D)}\log\|\wedge^{r_1(n,D)}f\|_v\leqslant\log\|f\|_v\leqslant\log\sqrt{r_1(n,D)},\]
  where the second inequality comes from the definition of metrics of John in \S \ref{metric of john}.

  Now we consider the case of $v\in M_{K,f}$. The homomorphism $f$ is induced by a homomorphism of $\O_K$-module \[F_D\rightarrow\bigoplus\limits_{i\in I}\mathcal P_i^*\O_{\mathbb P^n_{\O_K}}(D),\] where $\mathcal P_i$ is the Zariski closure of $P_i$ in $\mathscr X$ for each $i\in I$. Then for every $v\in M_{K,f}$, we have $\log\|\wedge^{r_1(n,D)}f\|_v\leqslant0$.

  We fix a maximal ideal $\p$ of $\O_K$ corresponding to $v\in M_{K,f}$, and decompose the set $\{\mathcal P_i\}_{i\in I}$ as the disjoint union
  \[\{\mathcal P_i\}_{i\in I}=\bigcup_{\xi\in \mathscr X(\f_\p)}\{\mathcal P_{l,\xi}\}_{l=1}^{m_\xi},\]
  where all elements in $\{\mathcal P_{l,\xi}\}_{l=1}^{m_\xi}$ modulo $\p$ are the same point $\xi\in\mathscr X(\f_\p)$. If $\{\mathcal P_{l,\xi}\}_{l=1}^{m_\xi}$ is empty for some $\xi\in\mathscr X(\f_\p)$, we define $m_\xi=0$ for simplicity. With the above notations, let
  \[\bigcup_{\xi\in\mathscr X(\f_\p)}\{s_{l,\xi}\}_{l=1}^{m_\xi}\]
  be an $\O_{K,\p}$-basis of $F_{D,\p}$ such that $f(s_{l,\xi})$ generates $\mathcal P_{l,\xi}^*\O_{\mathbb P^n_{\O_K}}(D)$ for all $l=1,\ldots,m_\xi$ and $\xi\in\mathscr X(\f_\p)$. Since $\O_{K,\p}$ is a local ring, the $\O_{K,\p}$-module $F_{D,\p}$ is free, then there exists such a basis for a fixed maximal ideal $\p$. We denote by $F_{D,\xi}$ the sub-$\O_{K,\p}$-module of $F_{D,\p}$ generated by $\{s_{l,\xi}\}_{l=1}^{m_\xi}$.

  By Proposition \ref{local control by Q_r}, we have
  \[\log\left\|\wedge^{\rg(F_{D,\xi})}f|_{F_{D,\xi}}\right\|_\p\leqslant -Q_\xi(\rg(F_{D,\xi}))\log N(\p).\]
  By definition, we have
  \[F_{D,\p}=\bigoplus_{\xi\in\mathscr X(\f_\p)}F_{D,\xi},\]
  and
  \begin{equation}\label{sum of dimensions of F_D}
    r_1(n,D)=\sum_{\xi\in\mathscr X(\f_\p)}\rg(F_{D,\xi}).
  \end{equation}
  Then from the above construction, by applying Lemma \ref{product of determinants} and Proposition \ref{local control by Q_r} respectively, we obtain
  \[\log\left\|\wedge^{r_1(n,D)}f\right\|_\p=\sum_{\xi\in\mathscr X(\f_\p)}\log\left\|\wedge^{\rg(F_{D,\xi})}f|_{F_{D,\xi}}\right\|_\p\leqslant-\sum\limits_{\xi\in\mathscr X(\f_\p)}Q_\xi\left(\rg(F_{D,\xi})\right)\log N(\p).\]

  In order to estimate the term \[\frac{1}{r_1(n,D)}\sum\limits_{\xi\in\mathscr X(\f_\p)}Q_\xi\left(\rg(F_{D,\xi})\right),\] by \eqref{lower bound of Q_r}, we have
  \begin{eqnarray*}
    & &\sum\limits_{\xi\in\mathscr X(\f_\p)}Q_\xi\left(\rg(F_{D,\xi})\right)\\
    &>&\sum_{\xi\in\mathscr X(\f_\p)}\left(\left(\frac{(n-1)!}{\mu_\xi(\mathscr X_\p)}\right)^{\frac{1}{n-1}}\left(\frac{n-1}{n}\right)\rg(F_{D,\xi})^{\frac{n}{n-1}}-\frac{n^3+2n^2+n-4}{2n(n+1)}\rg(F_{D,\xi})\right)\\
    &=&\frac{(n-1)!^{\frac{1}{n-1}}(n-1)}{n}\sum_{\xi\in\mathscr X(\f_\p)}\frac{\rg(F_{D,\xi})^{\frac{n}{n-1}}}{\mu_\xi(\mathscr X_\p)^{\frac{1}{n-1}}}-\frac{n^3+2n^2+n-4}{2n(n+1)}r_1(n,D).
  \end{eqnarray*}
  By \eqref{sum of dimensions of F_D} and H\"older's inequality, we have
  \[\sum_{\xi\in\mathscr X(\f_\p)}\frac{\rg(F_{D,\xi})^{\frac{n}{n-1}}}{\mu_\xi(\mathscr X_\p)^{\frac{1}{n-1}}}\geqslant\frac{\left(\sum\limits_{\xi\in\mathscr X(\f_\p)}\rg(F_{D,\xi})\right)^\frac{n}{n-1}}{\left(\sum\limits_{\xi\in\mathscr X(\f_\p)}\mu_\xi(\mathscr X_\p)\right)^\frac{1}{n-1}}=\frac{r_1(n,D)^\frac{n}{n-1}}{n(\mathscr X_\p)^\frac{1}{n-1}},\]
  where $n(\mathscr X_\p)$ is defined in the statement of Theorem \ref{global determinant}. Then we obtain the inequality
      \begin{eqnarray*}
       & &\;\;\frac{\wmu(\overline{F}_D)}{D}\leqslant\sup_{i\in I}h(P_i)+\frac{\log r_1(n,D)}{2D}\\
        & &-\frac{1}{[K:\Q]}\sum_{j\in J}\left(\frac{(n-1)!^{\frac{1}{n-1}}(n-1)r_1(n,D)^\frac{1}{n-1}}{n Dn(\mathscr X_{\p_j})^\frac{1}{n-1}}-\frac{n^3+2n^2+n-4}{2Dn(n+1)}\right)\log N(\p_j),
      \end{eqnarray*}
      which leads to a contradiction.
\end{proof}

\section{Some quantitative estimates}
In order to apply the global determinant method introduced in Theorem \ref{global determinant}, we need to gather enough information on the term $n(\mathscr X_{\p})$ in it. For this target, we need to have a control of the reduction type of $\mathscr X\hookrightarrow\mathbb P^n_{\O_K}\rightarrow\spec\O_K$, an upper bound of $n(\mathscr X_\p)$ when $\mathscr X_\p\rightarrow\spec\f_\p$ is geometrically integral, and a distribution of certain prime ideals of $\O_K$. We will also provide an explicit estimate of the geometric Hilbert-Samuel function of hypersurfaces.
\subsection{Control of the non-geometrically integral reductions}
Let $X\hookrightarrow \mathbb P^n_K$ be a geometrically integral hypersurface of degree $\delta$, $\mathscr X\hookrightarrow\mathbb P^n_{\O_K}\rightarrow\spec\O_K$ be its Zariski closure, and $\mathscr X_{\f_\p}=\mathscr X\times_{\spec\O_K}\spec\f_\p\rightarrow\spec\f_\p$ for every $\p\in\spm\O_K$. By \cite[Th\'eor\`eme 9.7.7]{EGAIV_3}, the set
\begin{equation}\label{non-geometricaly integral reduction}
  \mathcal Q(\mathscr X)=\left\{\p\in\spm\O_K|\;\mathscr X_{\f_\p}\rightarrow\spec\f_\p\hbox{ is not geometrically integral}\right\}
\end{equation}
is finite.

Next, we introduce a numerical description of the set $\mathcal Q(\mathscr X)$. In fact, there are fruitful results on this subject, but most of them are over rational number field $\Q$. In \cite{Liu-non-geometricallyintegral}, the estimate \cite[Satz 4]{Ruppert1986} was generalized over arbitrary number fields by using a height function in an adelic sense by the approach of \cite[\S3.4]{Liu-reduced}. By \cite[Proposition 4.1]{Liu-non-geometricallyintegral}, we have
\begin{equation}\label{non-geometrically integral reductions}
  \frac{1}{[K:\Q]}\sum_{\p\in\mathcal Q(\mathscr X)}\log N(\p)\leqslant(\delta^2-1)h(X)+C(n,\delta),
\end{equation}
where $h(X)$ is the naive height of $X$ defined in Definition \ref{classic height of hypersurface}, $N(\p)=\#(\O_K/\p)$, and the constant
    \begin{equation*}C(n,\delta)=(\delta^2-1)\left(3\log\delta+\delta\log3+\log{n+\delta\choose\delta}\right).
  \end{equation*}
  In fact, we have $C(n,\delta)\ll_n\delta^3$.

\subsection{Quantitative estimates over finite fields}
In this subsection, we give an upper bound of the term $n(\mathscr X_{\p})$ for an arbitrary maximal ideal $\p$ of $\O_K$, where $n(\mathscr X_{\p})$ is defined in the statement of Theorem \ref{global determinant}. In this part, we consider this problem over arbitrary finite fields.

Let $\f_q$ be the finite field with $q$ elements, $X$ be a geometrically integral hypersurface in $\mathbb P^n_{\f_q}$ of degree $\delta$, and $n(X_{\f_q})=\sum\limits_{\xi\in X(\f_q)}\mu_\xi(X)$, where $\mu_\xi(X)$ is the multiplicity of $\xi$ in $X$ defined via the local Hilbert-Samuel fuction in \eqref{definition of multiplicity}. Then we have
\[n(X_{\f_q})=\#X(\f_q)+\sum\limits_{\xi\in X(\f_q)}\left(\mu_\xi(X)-1\right).\]
In order to estimate $n(X_{\f_q})$, we will consider the terms $\#X(\f_q)$ and $\sum\limits_{\xi\in X(\f_q)}\left(\mu_\xi(X)-1\right)$ separately.

\subsubsection{} For the estimate of $\#X(\f_q)$, there are fruitful results on it. For our application, we have the following result deduced from \cite[Corollary 5.6]{cafure2006improved}.
\begin{prop}\label{number of rational points of hypersurface over finite field}
    Let $X\hookrightarrow\mathbb P^n_{\f_q}$ be a geometrically integral hypersurface of degree $\delta$ over the finite field $\f_q$. When $q\leqslant \delta^2$ or $q\geqslant27\delta^4$, we have
    \[\#X(\f_q)-q^{n-1}\leqslant n\delta^2q^{n-\frac{3}{2}}.\]
\end{prop}
\begin{proof}
  We consider this estimate case by case as following.
\begin{enumerate}
  \item If $q\leqslant\delta$, we have $\#X(\f_q)\leqslant\#\mathbb P^n(\f_q)=q^n+\cdots+1$. Then \[\#X(\f_q)-q^{n-1}\leqslant nq^n\leqslant n\delta^2q^{n-\frac{3}{2}}.\]
  \item If $\delta+1\leqslant q\leqslant\delta^2$, we have $\#X(\f_q)\leqslant\delta\#\mathbb P^{n-1}(\f_q)=\delta(q^{n-1}+\cdots+1)$. Then \[\#X(\f_q)-q^{n-1}\leqslant(\delta-1)q^{n-1}+\delta(q^{n-2}+\cdots+1)\leqslant n\delta^2q^{n-\frac{3}{2}}.\]
  \item If $q\geqslant27\delta^4$, by \cite[Corollary 5.6]{cafure2006improved}, we have \[\#X(\f_q)-q^{n-1}\leqslant (\delta-1)(\delta-2)q^{n-\frac{3}{2}}+(5\delta^2+\delta+1)q^{n-2}\leqslant n\delta^2q^{n-\frac{3}{2}}.\]
\end{enumerate}
\end{proof}
\begin{rema}
  With all the notations in Proposition \ref{number of rational points of hypersurface over finite field}. When $\delta^2\ll_nq\ll_n\delta^4$, by \cite[Corollary 5.6]{cafure2006improved}, we have
  \[\#X(\f_q)-q^{n-1}\leqslant(\delta-1)(\delta-2)q^{n-\frac{3}{2}}+B(n,\delta)q^{n-2},\]
  where the constant satisfies $B(n,\delta)\ll_n\delta^4$. It seems that the constant $B(n,\delta)$ could have a better dependence on $\delta$, but up to the author's knowledge, we do not know the answer.
\end{rema}
\subsubsection{}For the term $\sum\limits_{\xi\in X(\f_q)}\left(\mu_\xi(X)-1\right)$, by \cite[Theorem 5.1]{Liu-multiplicity}, we have
\begin{eqnarray}\label{estimate of multiplicities}
  \sum\limits_{\xi\in X(\f_q)}\left(\mu_\xi(X)-1\right)&\leqslant&\frac{1}{2}\sum\limits_{\xi\in X(\f_q)}\mu_\xi(X)\left(\mu_\xi(X)-1\right)\\
  &\leqslant&\frac{(n-1)^2}{2}\delta(\delta-1)\max\{\delta-1,q\}^{n-2},\nonumber
\end{eqnarray}
which has the optimal dependances on $\delta$ and $\max\{\delta-1,q\}$ when $q\geqslant\delta-1$.
\subsubsection{}
 We combine Propostion \ref{number of rational points of hypersurface over finite field} and the estimate \eqref{estimate of multiplicities}. When $q\leqslant \delta^2$ or $q\geqslant27\delta^4$, we have
\begin{equation*}
  n(X_{\f_q})\leqslant q^{n-1}+ n^2\delta^2\max\left\{q,\delta-1\right\}^{n-\frac{3}{2}}
\end{equation*}
by an elementary calculation. In addition, we have
\begin{equation}\label{estimate of sqrt[n-1]{n(X)}}
  \frac{1}{n(X_{\f_q})^{\frac{1}{n-1}}}\geqslant\frac{1}{q}-\frac{n^2\delta^2}{\max\left\{q,\delta-1\right\}^{\frac{3}{2}}}
\end{equation}
under the same assumption of $q$ and $\delta$ as above.

\subsection{Distribution of certain prime ideals}
In this part, we will consider some distributions of prime ideals of the ring of integers of number fields.
\subsubsection{Distribution of prime ideals containing a fixed ideal}
In this part, we will consider the distribution of certain maximal ideals of $\O_K$. First, we generalize \cite[Lemma 1.10]{Salberger_preprint2013} over an arbitrary number field, where the former result works over $\mathbb Z$ only.
\begin{lemm}\label{estimate of log p/p}
  Let $\mathfrak a$ be a proper ideal of $\O_K$, $\p$ be an prime ideal of $\O_K$, and $N(\mathfrak a)=\#(\O_K/\mathfrak a)$. Then we have
\[\frac{1}{[K:\Q]}\sum_{\p\supseteq\mathfrak a}\frac{\log N(\p)}{N(\p)}\leqslant\log\log\left(N(\mathfrak a)\right)+2,\]
  where the above sum takes all over the prime ideals contained in $\mathfrak a$ of $\O_K$.
\end{lemm}
\begin{proof}
We will prove the inequality for the case of $K=\Q$ at first, and then we show the general case by it.

  \textbf{Case of $K=\Q$. - }In this case, we will repeat the proof of \cite[Lemma 1.10]{Salberger_preprint2013} by Salberger, since this preprint is not easily available. Suppose that $\mathfrak a$ is generated by the positive square-free integer $\pi$, and let $m$ be a positive integer such that $m\leqslant\pi$. For the prime $p$, let $v_p(m)$ be the largest integer such that $p^{v_p(m)}\mid m$. By \cite[Tome I, Corollaire 1.7]{Ten} and \cite[Tome I, Th\'eor\`eme 1.8]{Ten}, we have
  \begin{equation*}
    m\sum_{p|\pi}\frac{\log p}{p}-\sum_{p|\pi}\log p\leqslant\sum_{p|\pi}v_p(m!)\log p\leqslant\sum_{p\leqslant\pi}v_p(m!)\log p=\log m!\leqslant m\log m,
  \end{equation*}
  and then we obtain
  \[\sum_{p|\pi}\frac{\log p}{p}\leqslant\log m+\frac{1}{m}\sum_{p|\pi}\log p\leqslant\log m+\frac{1}{m}\log \pi.\]
  Let $m=[\log\pi]$ for $\pi\geqslant2$, and then we accomplish the proof for $K=\Q$.

  \textbf{Case of arbitrary number fields. - }Let \[\mathfrak a=\p_1^{v_{\p_1}(\mathfrak a)}\cdots\p_k^{v_{\p_k}(\mathfrak a)},\] where $\p_1,\ldots,\p_k$ are distinct prime ideals of $\O_K$, and $v_{\p_i}(\mathfrak a)\in\mathbb N^+$ for all $i=1,\ldots,k$. Let the prime $p_i$ be the characteristic of the prime ideal $\p_i$, where $i=1,\ldots, k$ as above. For a fixed prime $p$, there are at most $[K:\Q]$ prime ideals of characteristic $p$ in $\O_K$. For all prime $p$ and $f\in\mathbb N^+$, we have
  \[\frac{\log p^f}{p^f}\leqslant\frac{\log p}{p}.\] Let $P(\mathfrak a)$ be the product of all the different characteristics of $\p_1,\ldots,\p_k$, and we have $P(\mathfrak a)\leqslant N(\mathfrak a)$ by definition directly. Then by the above facts, we obtain
\[\sum_{\p\supseteq\mathfrak a}\frac{\log N(\p)}{N(\p)}\leqslant[K:\Q]\sum_{p|P(\mathfrak a)}\frac{\log p}{p}.\]
  By the case of $K=\Q$, we have
  \[\sum_{p|P(\mathfrak a)}\frac{\log p}{p}\leqslant\log\log P(\mathfrak a)+2\leqslant\log\log N(\mathfrak a)+2,\]
  which proves the assertion.
\end{proof}
\subsubsection{Distribution of prime ideals with bounded norm}\label{implicit estimates of distribution of prime ideals of bounded norm}
Let $x\in\mathbb R^+$, $\p\in\spm\O_K$, and $N(\p)=\#(\O_K/\p)$. In this part, we consider the of

\begin{equation}\label{definition of theta_K psi_K phi_K}
\theta_K(x)=\sum\limits_{N(\p)\leqslant x}\log N(\p),\;\psi_K(x)=\sum\limits_{N(\p)\leqslant x}\frac{\log N(\p)}{N(\p)},\;\phi_K(x)=\sum\limits_{N(\p)\leqslant x}\frac{\log N(\p)}{N(\p)^\frac{3}{2}}.
\end{equation}

When $K=\Q$, these are classic estimates of Chebyshev function (cf. \cite[Tome I, Th\'eor\`eme 2.11]{Ten}) and Mertens' first theorem (cf. \cite[Tome I, Th\'eor\`eme 1.8]{Ten}). For the case of arbitrary number fields, a generalization of \cite[Tome I, Th\'eor\`eme 2.11]{Ten} is Landau's prime ideals theorem (cf. \cite[Theorem 2.2]{Rosen1999}), and a generalization of Mertens' first theorem was deduced from this in \cite[Lemma 2.3]{Rosen1999}.

In this part, we will give a more explicit version of some results \cite{Rosen1999}, which will be used in the application of Theorem \ref{global determinant}.

By Landau's prime ideal theorem (cf. \cite[Theorem 2.2]{Rosen1999}), we have
\[\theta_k(x)=x+O_K\left(xe^{-c\sqrt{\log x}}\right),\]
where $c$ is a constant depending on $K$. Then there exists a function $\epsilon_1(K,x)$ of the number field $K$ and $x\in\mathbb R^+$, such that
\begin{equation}\label{implicit estimate of theta_K}
  \left|\theta_K(x)-x\right|\leqslant\epsilon_1(K,x),
\end{equation}
  where $\epsilon_1(K,x)=O_K\left(xe^{-c\sqrt{\log x}}\right)$ for all $x\in\mathbb R^+$, and $c$ depends on $K$ only.

By \cite[Lemma 2.3]{Rosen1999}, we have
\[\psi_K(x)=\log x+O_K(1),\]
which is obtained by Abel's summation formula applied in \cite[Lemma 2.1]{Rosen1999}. Then there exists a function $\epsilon_2(K)$ of the number field $K$, such that
\begin{equation}\label{implicit estimate of psi_K}
  \left|\psi_K(x)-\log x\right|\leqslant\epsilon_2(K).
\end{equation}

Same as the application of \cite[Lemma 2.1]{Rosen1999} to the proof of \cite[Lemma 2.3]{Rosen1999}, we have
\[\phi_K(x)=\frac{\theta_K(x)}{x^\frac{3}{2}}+\frac{3}{2}\int_2^x\frac{\theta_K(t)}{t^\frac{5}{2}}dt\]
by \cite[Lemma 2.1]{Rosen1999}. Then by \eqref{implicit estimate of theta_K}, we have
\[\left|\phi_K(x)-\frac{1}{\sqrt{x}}-\frac{3}{2}\int_2^x\frac{1}{t^{\frac 32}}dt\right|\leqslant\frac{\epsilon_1(K,x)}{x^{\frac 32}}+\frac{3}{2}\int_2^x\frac{\epsilon_1(K,t)}{t^{\frac 52}}dt.\]
Then by an elementary calculation, there exists a function $\epsilon_3(K,x)$ of the number field $K$ and $x\in\mathbb R^+$, such that
\begin{equation}\label{implicit estimate of phi_K}
  \left|\phi_K(x)-\frac{3}{2}\sqrt{2}+\frac{2}{\sqrt{x}}\right|\leqslant \epsilon_3(K,x),
\end{equation}
 where
 \begin{equation}\label{definition of epsilon_3}
 \epsilon_3(K,x)=\frac{\epsilon_1(K,x)}{x^{\frac 32}}+\frac{3}{2}\int_2^x\frac{\epsilon_1(K,t)}{t^{\frac 52}}dt.
\end{equation}

\subsubsection{Distribution of non-geometrically integral reductions}
In this part, we consider a distribution of the prime ideals modulo which the reductions are not geometrically integral. In the estimate below, the estimates of \eqref{definition of theta_K psi_K phi_K} will be involved.

\begin{prop}\label{estimate of log p/p for non-geometrically integral 2}
  Let $X$ be a geometrically integral hypersurface of degree $\delta$ in $\mathbb P^n_K$, and $\mathscr X$ be its Zariski closure in $\mathbb P^n_{\O_K}$. Let $\p\in\spm\O_K$, $N(\p)=\#(\O_K/\p)$,
  \[\mathcal Q'(\mathscr X)=\{\p\in\spm\O_K|\;N(\p)>27\delta^4\hbox{ and }\mathscr X_{\p}\rightarrow\spec\f_\p\hbox{ not geometrically integral}\},\]
and
\[b'(\mathscr X)=\prod_{\p\in\mathcal Q'(\mathscr X)}\exp\left(\frac{\log N(\p)}{N(\p)}\right).\]
Then we have
\begin{eqnarray*}
  b'(\mathscr X)&\leqslant&\exp\left(2\epsilon_2(K)-3\log3+[K:\Q]\right)\\
  & &\;\cdot(\delta^{-2}-\delta^{-4})\left(h(X)+\left(3\log\delta+\delta\log3+\log{n+\delta\choose\delta}\right)\right),
\end{eqnarray*}
where $h(X)$ is defined in Definition \ref{classic height of hypersurface}, and $\epsilon_2(K)$ is defined in \eqref{implicit estimate of psi_K}.
\end{prop}
\begin{proof}
  We denote by $\mathcal P'(\mathscr X)$ the product of all maximal ideals in $\mathcal Q'(\mathscr X)$, and \[c'(\mathscr X)=(\delta^2-1)\left(h(X)+\left(3\log\delta+\delta\log3+\log{n+\delta\choose\delta}\right)\right).\] Then by Lemma \ref{estimate of log p/p} and \eqref{non-geometrically integral reductions}, we have
  \begin{eqnarray*}
    & &\frac{1}{[K:\Q]}\log b'(\mathscr X)=\frac{1}{[K:\Q]}\sum_{\p\in\mathcal Q'(\mathscr X)}\frac{\log N(\p)}{N(\p)}\\
    &\leqslant&\frac{1}{[K:\Q]}\sum_{27\delta^4<N(\p)\leqslant c'(\mathscr X)}\frac{\log N(\p)}{N(\p)}+\frac{1}{[K:\Q]}\sum_{\begin{subarray}{c}\p\in\mathcal Q'(\mathscr X)\\N(\p)>c'(\mathscr X)\end{subarray}}\frac{\log N(\p)}{ c'(\mathscr X)}.
  \end{eqnarray*}
  By \eqref{implicit estimate of psi_K}, we have
  \begin{eqnarray*}
    & &\frac{1}{[K:\Q]}\sum_{27\delta^4<N(\p)\leqslant c'(\mathscr X)}\frac{\log N(\p)}{N(\p)}\\
    &=&\frac{1}{[K:\Q]}\sum_{N(\p)\leqslant c'(\mathscr X)}\frac{\log N(\p)}{N(\p)}-\frac{1}{[K:\Q]}\sum_{N(\p)\leqslant27\delta^4}\frac{\log N(\p)}{N(\p)}\\
    &\leqslant&\frac{1}{[K:\Q]}\left(\log c'(\mathscr X)-4\log\delta+2\epsilon_2(K)-3\log3\right),
  \end{eqnarray*}
where $\epsilon_2(K)$ is defined in \eqref{implicit estimate of psi_K} depending on $K$ only.

Let
\[\mathcal Q(\mathscr X)=\left\{\p\in\spm\O_K|\;\mathscr X_{\f_\p}\rightarrow\spec\f_\p\hbox{ is not geometrically integral}\right\}.\]
Since $\mathcal Q'(\mathscr X)\subseteq\mathcal Q(\mathscr X)$, then we have
\begin{equation*}
  \frac{1}{[K:\Q]}\sum_{\begin{subarray}{c}\p\in\mathcal Q'(\mathscr X)\\N(\p)>c'(\mathscr X)\end{subarray}}\frac{\log N(\p)}{ c'(\mathscr X)}\leqslant\frac{1}{[K:\Q]c'(\mathscr X)}\sum_{\p\in\mathcal Q(\mathscr X)}\log N(\p)\leqslant1,
\end{equation*}
where the last inequality is from \eqref{non-geometrically integral reductions}. By combining the above two estimates, we terminate the proof.
\end{proof}
\begin{rema}
  With all the notations and assumptions in Proposition \ref{estimate of log p/p for non-geometrically integral 2}. We have
  \[b'(\mathscr X)\ll_{n,K}\max\left\{\delta^{-2}h(X),\delta^{-1}\right\}.\]
\end{rema}

\subsection{An explicit estimate of the geometric Hilbert-Samuel function}
In this part, we will provide an explicit lower bound of the geometric Hilbert-Samuel function of a projective hypersurface, which will be used in the application of the determinant method. The inequality
 \[\frac{(N-m+1)^m}{m!}\leqslant{N \choose m}\leqslant \frac{\left(N-(m-1)/2\right)^m}{m!}\]
will be helpful in the calculation below.
\begin{lemm}\label{explicit bounds of geometric Hilbert-Samuel function}
  Let $X$ be a hypersurface of degree $\delta$ in $\mathbb P^n_K$. We denote by $r_1(n,D)$ its geometric Hilbert-Samuel function with the variable $D$. When $D\geqslant\delta+1$, we have
  \[r_1(n,D)^\frac{1}{n-1}\geqslant\sqrt[n-1]{\frac{\delta}{(n-1)!}}D-(\delta-2)\sqrt[n-1]{\frac{\delta}{(n-1)!}},\]
  and
  \[r_1(n,D)^\frac{1}{n-1}\leqslant\sqrt[n-1]{\frac{\delta}{(n-1)!}}D+\frac{n}{2}\sqrt[n-1]{\frac{\delta}{(n-1)!}}.\]

\end{lemm}
\begin{proof}
  In fact, we have
  \[r_1(n,D)={n+D\choose n}-{n+D-\delta\choose n}\]
  when $D\geqslant\delta+1$.

  In order to obtain the lower bound, we have
  \begin{eqnarray}\label{lower bound of r_1(n,D)}
    r_1(n,D)&=&\sum_{j=1}^\delta{D-\delta+n-1+j\choose n-1}\geqslant\frac{\delta(D-\delta+2)^{n-1}}{(n-1)!}\\
    &\geqslant&\frac{\delta}{(n-1)!}D^{n-1}-\frac{\delta(\delta-2)}{(n-1)!}D^{n-2}.\nonumber
  \end{eqnarray}
  Then we obtain
  \begin{eqnarray*}
    r_1(n,D)^\frac{1}{n-1}&\geqslant&\sqrt[n-1]{\frac{\delta}{(n-1)!}}D\left(1-\frac{\delta-2}{D}\right)^\frac{1}{n-1}\\
    &\geqslant&\sqrt[n-1]{\frac{\delta}{(n-1)!}}D-(\delta-2)\sqrt[n-1]{\frac{\delta}{(n-1)!}}
  \end{eqnarray*}
when $D\geqslant\delta+1$.

On the other hand, we have
  \begin{equation}\label{upper bound of r_1(n,D)}
    r_1(n,D)=\sum_{j=1}^\delta{D-\delta+n-1+j\choose n-1}\leqslant\frac{\delta(D+\frac{n}{2})^{n-1}}{(n-1)!}=\frac{\delta D^{n-1}}{(n-1)!}\left(1+\frac{n}{2D}\right)^{n-1},
  \end{equation}
  which terminates the proof by an elementary calculation.
\end{proof}
\section{An explicit estimate of determinant}
In this section, we will give an upper bound of the degree of the auxiliary hypersurface determined by Theorem \ref{global determinant}.
\subsection{A uniform lower bound of arithmetic Hilbert-Samuel functions}
Firstly, we refer to a result in \cite{Chen1}, which is an application of the uniform lower bound of the arithmetic Hilbert-Samuel functions to the determinant method.
\begin{prop}[\cite{Chen1}, Propoosition 2.12]\label{evaluation map of points}
 We keep all the notations in \S\ref{basic setting}. Let $X$ be a closed integral subscheme of $\mathbb P^n_K$, $Z=(P_i)_{i\in I}$ be a family of rational points, and
 \[\phi_{Z,D}:F_{D,K}\rightarrow\bigoplus_{i\in I}P^*_i\O_{\mathbb P^n_K}(D)\]
 be the evaluation map. If we have the inequality
  \[\sup_{i\in I}h_{\overline{\O(1)}}(P_i)<\frac{\wmu_{\max}(\F_D)}{D}-\frac{1}{2D}\log r_1(n,D),\]
  where $r_1(n,D)=\rg(F_{D})$ and the height function $h_{\overline{\O(1)}}(\ndot)$ is defined in \eqref{arakelov height}, then the homomorphism $\phi_{Z,D}$ is not able to be injective.
\end{prop}
The uniform lower bound of $\wmu(\overline F_D)$ for all $D\geqslant1$ will play a significant role in the construction of auxiliary hypersurfaces if we want to apply Proposition \ref{evaluation map of points}. In \cite{David_Philippon99}, David and Philippon give an explicit uniform lower bound of $\wmu(\overline F_D)$. This result is reformulated by H. Chen in \cite[Theorem 4.8]{Chen1} by the language of the slope method. In fact, let $X$ be a closed integral subscheme of $\mathbb P^n_K$ of dimension $d$ and degree $\delta$, and $\mathscr X$ be its Zariski closure in $\mathbb P^n_{\O_K}$. The inequality
\begin{equation}\label{uniform lower bound of arithmetic Hilbert-Samuel}
  \frac{\wmu(\F_D)}{D}\geqslant\frac{d!}{\delta(2d+2)^{d+1}}h_{\overline{\O(1)}}(\mathscr X)-\log(n+1)-2^d
\end{equation}
is uniformly verified for all $D\geqslant2(n-d)(\delta-1)+d+2$ (see also \cite[Remark 4.9]{Chen1} for some minor modifications), where $h_{\overline{\O(1)}}(\mathscr X)$ follows the definition in Definition \ref{arakelov height of projective variety}.

By Proposition \ref{evaluation map of points}, all the rational points with small heights in a projective variety can be covered by one hypersurface which does not contain the generic point of the original variety. The following result gives a numerical description of this observation.
\begin{prop}\label{naive siegal lemma}
  Let $X$ be an integral hypersurface of degree $\delta$ in $\mathbb P^n_K$, the constant $\mathcal H_n=1+\cdots+\frac{1}{n}$ and the constant
  \begin{equation}\label{constant C_1}
    C_1(n)=-\frac{(2n)^n}{(n-1)!}\left(\log 2+5\log(n+1)-\frac{1}{2}\mathcal H_n\right)-\frac{3}{2}\log(n+1)-2^{n-1}.
  \end{equation}
 If
  \[\frac{\log B}{[K:\Q]}<\frac{(n-1)!}{\delta(2n)^n}h(X)+C_1(n),\]
  then there exists a hypersurface of degree smaller than $2\delta+n-1$, which contains all rational points in $S(X;B)$ but does not contain the generic point of $X$, where we use the height function defined in \eqref{arakelov height}.
\end{prop}
\begin{proof}
  If there does not exist such a hypersurface, the evaluation map $\phi_{Z,D}$ in Proposition \ref{evaluation map of points} is injective. On the other hand, by \eqref{uniform lower bound of arithmetic Hilbert-Samuel}, Proposition \ref{comparing heights} and the fact that
  \[r_1(n,D)\leqslant{n+D\choose n}\leqslant(n+1)^D\]
  is uniformly verified for all $n, D\geqslant1$, we obtain
  \begin{eqnarray*}
    \frac{\log B}{[K:\Q]}&<&\frac{(n-1)!}{\delta(2n)^n}h(X)+C_1(n)\\
    &\leqslant&\frac{(n-1)!}{\delta(2n)^n}h_{\overline{\O(1)}}(\mathscr X)-\frac{3}{2}\log(n+1)-2^{n-1}\\
    &\leqslant&\frac{\wmu_{\max}(\F_D)}{D}-\frac{1}{2D}\log r_1(n,D),
  \end{eqnarray*}
  which contradicts to Proposition \ref{evaluation map of points}.
  \end{proof}
  \begin{rema}
    With all the notations in Proposition \ref{naive siegal lemma}. By the arithmetic Hilbert-Samuel Theorem of arithmetic ample line bundles (cf. \cite[Theorem 8]{Gillet-Soule}, \cite[Theorem 1.4]{Zhang95} and \cite[Th\'eor\`eme principal]{Abbes-Bouche}), we have
    \[\wmu(\overline F_D)=\frac{h_{\overline{\O(1)}}(\mathscr X)}{n\delta}D+o(D)\]
    for $D$ tends into infinity. So it is expected that we can obtain a better uniform lower bound of $\wmu(\F_D)$ than that in \eqref{uniform lower bound of arithmetic Hilbert-Samuel}. If we have a better explicit lower bound, we can improve the bound given in Proposition \ref{naive siegal lemma}.
  \end{rema}
  \subsection{Estimate of the determinant}
  In the global determinant method, for each geometrically integral hypersurface, we allow only one auxiliary hypersurface to cover its rational points with bounded height not containing the generic point of the original hypersurface, and we optimize the degree of this auxiliary hypersurface.

   In this part, we will give an upper bound of the degree of the auxiliary hypersurface determined in Theorem \ref{global determinant}, where the size of non-geometrically integral reductions and the height of the original hypersurface will be involved.

   Before the statement of the main theorem in this paragraph, we will introduce two constants depending on the number field $K$ and the positive integer $n\geqslant2$, which will be used in the estimate of the determinant.

   Let $K$ be a number field, we denote
   \begin{equation}\label{kappa_1(K)}
     \kappa_1(K)=\sup_{x\in\mathbb R^+}\frac{\epsilon_1(K,x)}{x},
   \end{equation}
   where $\epsilon_1(K,x)$ is introduced in \eqref{implicit estimate of theta_K}. By \cite[Theorem 2.2]{Rosen1999}, the above supremum exists and $\kappa_1(K)$ depends on $K$ only.

   Let $\delta\geqslant1$ and $n\geqslant2$ be two integers, we denote
   \begin{equation}\label{kappa_2(K)}
       \kappa_2(K,n)=\sup_{\delta\geqslant1}\left\{-3\log3+\frac{2n^2}{3\sqrt{3}}+2n^2\delta^2\epsilon_3(K,27\delta^4)+2\epsilon_2(K)\right\},
\end{equation}
where $\epsilon_2(K)$ is defined in \eqref{implicit estimate of psi_K}, and $\epsilon_3(K,x)$ is defined in \eqref{definition of epsilon_3}. By taking \cite[Theorem 2.2]{Rosen1999} into the estimate of \eqref{definition of epsilon_3}, the above supremum exists and $\kappa_2(K,n)$ depends on $K$ and $n$ only.
  \begin{theo}\label{upper bound of degree of general case}
    Let $K$ be a number field. Let $X$ be a geometrically integral hypersurface in $\mathbb P^n_K$ of degree $\delta\geqslant2$, and $S(X;B)$ be the set of rational points in $X$ whose height is smaller than $B$ with respect to the above closed immersion, see \eqref{arakelov height} for the definition of the height function used above. Then there exists a hypersurface in $\mathbb P^n_K$ of degree smaller than
    \[e^{C_2(n,K)}B^{{n}/{\left((n-1)\delta^{1/(n-1)}\right)}}\delta^{4-{1}/{(n-1)}}\frac{b'(\mathscr X)}{H_K(X)^{\frac{n!}{(n-1)(2n)^n}\delta^{-1-{1}/{(n-1)}}}}\]
    which covers $S(X;B)$ but does not contain the generic point of $X$, where the constant
    \begin{eqnarray*}
     C_2(n,K)&=&\frac{nC_1(n)[K:\Q]}{(n-1)\sqrt[n-1]{2}}+\kappa_2(K,n)+\frac{\log(n-1)!}{n-1}\\
    & & +3+\frac{n^3+2n^2+n-4}{2(n^2-1)\sqrt[n-1]{(n-1)!}}\left(1+\frac{n}{4}\right)\left(1+\kappa_1(K)\right),
    \end{eqnarray*}
    the constant $C_1(n)$ is defined in \eqref{constant C_1}, $b'(\mathscr X)$ is defined in Proposition \ref{estimate of log p/p for non-geometrically integral 2}, $\kappa_1(K)$ is defined in \eqref{kappa_1(K)}, $\kappa_2(K,n)$ is defined in \eqref{kappa_2(K)}, and the height $H_K(X)$ of $X$ is defined in Definition \ref{classic height of hypersurface}.
  \end{theo}
  \begin{proof}By Proposition \ref{naive siegal lemma}, we divide the proof into two parts.

  \textbf{I. Case of large height varieties. - }If
    \[\frac{\log B}{[K:\Q]}<\frac{(n-1)!}{\delta(2n)^n}h(X)+C_1(n),\]
where the constant $C_1(n)$ is defined in \eqref{constant C_1} and $h(X)$ is defined in Definition \ref{classic height of hypersurface}. Then by Proposition \ref{naive siegal lemma}, $S(X;B)$ can be covered by a hypersurface of degree no more than $2\delta+n-1$ which does not contain the generic point of $X$. By an elementary calculation, we obtain that $2\delta+n-1$ is smaller than the bound provided in the statement of the theorem, for $n\geqslant2$, $\delta\geqslant1$, $b'(\mathscr X)\geqslant1$, and $H_K(X)\geqslant1$.

\textbf{II. Case of small height varieties. - }For the case of
\[\frac{\log B}{[K:\Q]}\geqslant\frac{(n-1)!}{\delta(2n)^n}h(X)+C_1(n),\]
which is equivalent to
\[h(X)\leqslant \frac{\delta(2n)^n}{(n-1)!}\cdot\left(\frac{\log B}{[K:\Q]}-C_1(n)\right),\]
   we will treat it as following. We keep all the notations in Theorem \ref{global determinant}, and we suppose $D\geqslant3\delta\log\delta+n-1\geqslant2\delta+n-1$ from now on. We denote the set
    \begin{eqnarray*}
      \mathcal R(\mathscr X)&=&\{\p\in\spm\O_K|\;27\delta^4\leqslant N(\p)\leqslant r_1(n,D)^\frac{1}{n-1},\\
      & &\;\mathscr X_{\p}\rightarrow\spec\f_\p\hbox{ is geometrically integral}\},
    \end{eqnarray*}
    and we apply Theorem \ref{global determinant} to the reductions at $\mathcal R(\mathscr X)$. If there does not exist such a hypersurface, then by Theorem \ref{global determinant} applied in the above sense, we have
    \begin{eqnarray*}
      &&\frac{\log B}{[K:\Q]}\geqslant\frac{\wmu(\F_D)}{D}-\frac{\log r_1(n,D)}{2D}\\
      &&\;+\frac{1}{[K:\Q]}\sum_{\p\in \mathcal R(\mathscr X)}\left(\frac{(n-1)!^{\frac{1}{n-1}}(n-1)r_1(n,D)^\frac{1}{n-1}}{nDn(\mathscr X_{\p})^\frac{1}{n-1}}-\frac{n^3+2n^2+n-4}{2Dn(n+1)}\right)\log N(\p).
    \end{eqnarray*}
    From the explicit lower bound of $\wmu(F_D)$ provided at \eqref{uniform lower bound of arithmetic Hilbert-Samuel} and Proposition \ref{naive siegal lemma}, we deduce
    \begin{eqnarray}\label{lower bound of the determinant}
    & &\frac{\log B}{[K:\Q]}-\frac{(n-1)!}{\delta(2n)^n}h(X)+C_1(n)\\
    &\geqslant&\frac{(n-1)!^{\frac{1}{n-1}}(n-1)r_1(n,D)^\frac{1}{n-1}}{nD[K:\Q]}\sum_{\p\in\mathcal R(\mathscr X)}\frac{\log N(\p)}{n(\mathscr X_{\p})^\frac{1}{n-1}}\nonumber\\
    & &\;-\frac{n^3+2n^2+n-4}{2nD(n+1)[K:\Q]}\sum_{\p\in\mathcal R(\mathscr X)}\log N(\p).\nonumber\end{eqnarray}

    \textbf{II-1. Estimate of $\sum\limits_{\p\in\mathcal R(\mathscr X)}\frac{\log N(\p)}{n(\mathscr X_{\p})^\frac{1}{n-1}}$. - } In order to estimate $\sum\limits_{\p\in\mathcal R(\mathscr X)}\frac{\log N(\p)}{n(\mathscr X_{\p})^\frac{1}{n-1}}$ in \eqref{lower bound of the determinant}, by \eqref{estimate of sqrt[n-1]{n(X)}}, we have
    \[\sum\limits_{\p\in\mathcal R(\mathscr X)}\frac{\log N(\p)}{n(\mathscr X_{\p})^\frac{1}{n-1}}\geqslant\sum_{\p\in\mathcal R(\mathscr X)}\frac{\log N(\p)}{N(\p)}-n^2\delta^2\sum_{\p\in\mathcal R(\mathscr X)}\frac{\log N(\p)}{N(\p)^\frac{3}{2}}.\]

    For the estimate of $\sum\limits_{\p\in\mathcal R(\mathscr X)}\frac{\log N(\p)}{N(\p)}$, we denote
    \begin{eqnarray*}
      \mathcal Q'(\mathscr X)&=&\{\p\in\spm\O_K|\;27\delta^4\leqslant N(\p)\leqslant r_1(n,D)^\frac{1}{n-1},\\
      & &\;\mathscr X_{\f_\p}\rightarrow\spec\f_\p\hbox{ is not geometrically integral}\}.
    \end{eqnarray*}
   Then by \eqref{implicit estimate of psi_K}, we have
    \begin{eqnarray*}
      & &\sum_{\p\in\mathcal R(\mathscr X)}\frac{\log N(\p)}{N(\p)}=\sum_{27\delta^4\leqslant N(\p)\leqslant r_1(n,D)^\frac{1}{n-1}}\frac{\log N(\p)}{N(\p)}-\sum_{\p\in\mathcal Q'(\mathscr X)}\frac{\log N(\p)}{N(\p)}\\
      &\geqslant&\frac{1}{n-1}\log r_1(n,D)-3\log3-4\log\delta-2\epsilon_2(K)-\log\left(b'(\mathscr X)\right),
    \end{eqnarray*}
 where the notation $b'(\mathscr X)$ is introduced in Proposition \ref{estimate of log p/p for non-geometrically integral 2}, and $\epsilon_2(K)$ is defined in \eqref{implicit estimate of psi_K}.

    For the term $\sum\limits_{\p\in\mathcal R(\mathscr X)}\frac{\log N(\p)}{N(\p)^\frac{3}{2}}$, it is equal to zero when $r_1(n,D)^{\frac{1}{n-1}}\leqslant27\delta^4$. When $r_1(n,D)^{\frac{1}{n-1}}>27\delta^4$, by \eqref{implicit estimate of phi_K}, we have
    \begin{eqnarray*}
      \sum\limits_{\p\in\mathcal R(\mathscr X)}\frac{\log N(\p)}{N(\p)^\frac{3}{2}}&\leqslant&\sum\limits_{27\delta^4\leqslant N(\p)\leqslant r_1(n,D)^{\frac{1}{n-1}}}\frac{\log N(\p)}{N(\p)^\frac{3}{2}}\\
      &\leqslant&\frac{2}{3\sqrt{3}\delta^2}-2r_1(n,D)^{-\frac{1}{2(n-1)}}+2\epsilon_3(K,27\delta^4).
    \end{eqnarray*}

    By the above two estimates, we obtain
    \begin{eqnarray}\label{estimate of sum N_p/log N_p}
      & &\sum\limits_{\p\in\mathcal R(\mathscr X)}\frac{\log N(\p)}{n(\mathscr X_{\p})^\frac{1}{n-1}}\\
      &\geqslant&\frac{1}{n-1}\log r_1(n,D)-3\log3-4\log\delta-\log\left(b'(\mathscr X)\right)-2\epsilon_2(K)\nonumber\\
      & &\;-n^2\delta^2\left(\frac{2}{3\sqrt{3}\delta^2}-2r_1(n,D)^{-\frac{1}{2(n-1)}}+2\epsilon_3(K,27\delta^4)\right)\nonumber
    \end{eqnarray}
    by combining the above two inequalities.

   \textbf{II-2. Estimate of $\sum\limits_{\p\in\mathcal R(\mathscr X)}\log N(\p)$. - } For the estimate of $\sum\limits_{\p\in\mathcal R(\mathscr X)}\log N(\p)$, by \eqref{implicit estimate of theta_K}, we have
    \begin{eqnarray}\label{estimate of sum log N_p0}
      & &\frac{1}{D}\sum\limits_{\p\in\mathcal R(\mathscr X)}\log N(\p)\leqslant\frac{1}{D}\sum\limits_{N(\p)\leqslant r_1(n,D)^\frac{1}{n-1}}\log N(\p)\\
      &\leqslant&\frac{1}{D}\left(r_1(n,D)^\frac{1}{n-1}+\epsilon_1\left(K,r_1(n,D)^\frac{1}{n-1}\right)\right),\nonumber
    \end{eqnarray}
    where $\epsilon_1(K,x)$ is defined in \eqref{implicit estimate of theta_K}.

    \textbf{II-3. Deducing the contradiction. - } We take \eqref{estimate of sum N_p/log N_p} and \eqref{estimate of sum log N_p0} into \eqref{lower bound of the determinant}, and we do some elementary calculations. Then the inequality
       \begin{eqnarray}\label{lower bound of the determinant2}
    & &\frac{\log B}{[K:\Q]}-\frac{(n-1)!}{\delta(2n)^n}h(X)+C_1(n)\\
    &\geqslant&\frac{(n-1)!^{\frac{1}{n-1}}(n-1)}{n[K:\Q]}\cdot\frac{r_1(n,D)^\frac{1}{n-1}}{D}\Bigg(\frac{1}{n-1}\log r_1(n,D)-\log\left(b'(\mathscr X)\right)-4\log\delta\nonumber\\
      & &\;3\log3-\frac{2n^2}{3\sqrt{3}}+\frac{2n^2\delta^2}{r_1(n,D)^{\frac{1}{2(n-1)}}}-2n^2\delta^2\epsilon_3(K,27\delta^4)-2\epsilon_2(K)\Bigg)\nonumber\\
    & &\;-\frac{n^3+2n^2+n-4}{2n(n+1)[K:\Q]}\cdot\frac{r_1(n,D)^\frac{1}{n-1}}{D}\left(1+\frac{\epsilon_1\left(K,r_1(n,D)^\frac{1}{n-1}\right)}{r_1(n,D)^{\frac{1}{n-1}}}\right)\nonumber\end{eqnarray}
is uniformly verified for all $D\geqslant3\delta\log\delta+n-1\geqslant2\delta+n-1$.

      From \eqref{lower bound of r_1(n,D)} in Lemma \ref{explicit bounds of geometric Hilbert-Samuel function}, we have
    \begin{equation}\label{lower bound of log r_1(n,D)}
     \frac{1}{n-1} \log r_1(n,D)\geqslant\log D+\frac{1}{n-1}\log\delta-\frac{1}{n-1}\log\left(n-1\right)!
    \end{equation}
    when $D\geqslant2\delta+n-1$.

    We take Lemma \ref{explicit bounds of geometric Hilbert-Samuel function} and \eqref{lower bound of log r_1(n,D)} into \eqref{lower bound of the determinant2}, and by the fact
    \[\frac{\epsilon_1\left(K,r_1(n,D)^\frac{1}{n-1}\right)}{r_1(n,D)^{\frac{1}{n-1}}}\leqslant\kappa_1(K)\]
    and
    \[3\log3-\frac{2n^2}{3\sqrt{3}}+\frac{2n^2\delta^2}{r_1(n,D)^{\frac{1}{2(n-1)}}}-2n^2\delta^2\epsilon_3(K,27\delta^4)-2\epsilon_2(K)\geqslant-\kappa_2(n,K),\]
     we obtain
           \begin{eqnarray}\label{lower bound of the determinant3}
    & &\frac{n}{(n-1)\sqrt[n-1]{\delta}}\left(\log B-\frac{[K:\Q](n-1)!}{\delta(2n)^n}h(X)+C_1(n)[K:\Q]\right)\\
    &\geqslant&\left(1-\frac{\delta-2}{D}\right)\Bigg(\log D-\left(4-\frac{1}{n-1}\right)\log\delta-\log\left(b'(\mathscr X)\right)\nonumber\\
      & &\;-\kappa_2(n,K)-\frac{\log(n-1)!}{n-1}\Bigg)\nonumber\\
    & &\;-\frac{n^3+2n^2+n-4}{2(n^2-1)\sqrt[n-1]{(n-1)!}}\cdot\left(1+\frac{n}{4}\right)\left(1+\kappa_1(K)\right).\nonumber\end{eqnarray}

    When $D\geqslant 3\delta\log\delta+n-1\geqslant2\delta+n-1$ and $\delta\geqslant2$, we have
    \begin{eqnarray*}
   & &\frac{\delta-2}{D}\Bigg(\log D-\left(4-\frac{1}{n-1}\right)\log\delta-\frac{\log(n-1)!}{n-1}-\log\left(b'(\mathscr X)\right)-\kappa_2(K,n)\Bigg)\\
   &\leqslant&\frac{\delta-2}{D}\log D\leqslant3
    \end{eqnarray*}
    by an elementary calculation. We take the above inequality into \eqref{lower bound of the determinant3}, and then we obtain
               \begin{eqnarray*}
    & &\frac{n}{(n-1)\sqrt[n-1]{\delta}}\left(\log B-\frac{[K:\Q](n-1)!}{\delta(2n)^n}h(X)+C_1(n)[K:\Q]\right)\\
    &\geqslant&\log D-\left(4-\frac{1}{n-1}\right)\log\delta-\log\left(b'(\mathscr X)\right)-\kappa_2(K,n)-3-\frac{\log(n-1)!}{n-1}\nonumber\\
    & &\;-\frac{n^3+2n^2+n-4}{2(n^2-1)\sqrt[n-1]{(n-1)!}}\cdot\left(1+\frac{n}{4}\right)\left(1+\kappa_1(K)\right),\nonumber\end{eqnarray*}
    which deduces
    \begin{eqnarray*}
      \log D&\leqslant& \frac{n\log B}{(n-1)\sqrt[n-1]{\delta}}-\frac{[K:\Q]n!}{\delta^{1+\frac{1}{n-1}}(n-1)(2n)^n}h(X)+\log\left(b'(\mathscr X)\right)\\
      & &+\left(4-\frac{1}{n-1}\right)\log\delta+C_2(n,K)
    \end{eqnarray*}
    with the constant $C_2(n,K)$ in the statement of this theorem, and it leads to the contradiction.
  \end{proof}
    \subsection{Control of auxiliary hypersurfaces}
  The following two upper bounds of the degree of the auxiliary hypersurface are deduced from Theorem \ref{upper bound of degree of general case} directly.
  \begin{coro}\label{uniform upper bound of the degree of auxiliary hypersurface}
    We keep all the notations and conditions in Theorem \ref{upper bound of degree of general case}. Then there exists a hypersurface of degree smaller than
     \[C_3(n,K)\delta^3B^{{n}/{\left((n-1)\delta^{1/(n-1)}\right)}},\]
     which covers $S(X;B)$ but does not contain the generic point of $X$, where the constant
\[C_3(n,K)=e^{C_2(n,K)}\frac{(n+6)(n-1)(2n)^n\exp\left(2\epsilon_2(K)-3\log3+[K:\Q]\right)}{n!},\]
and $C_2(n,K)$ is defined in Theorem \ref{upper bound of degree of general case}.
  \end{coro}
  \begin{proof}
    By the upper bound of $b'(\mathscr X)$ given in Proposition \ref{estimate of log p/p for non-geometrically integral 2}, we have
\begin{eqnarray*}
  b'(\mathscr X)&\leqslant&\exp\left(2\epsilon_2(K)-3\log3+[K:\Q]\right)\\
  & &\cdot(\delta^{-2}-\delta^{-4})\left(h(X)+\left(3\log\delta+\delta\log3+\log{n+\delta\choose\delta}\right)\right)\\
  &\leqslant&\exp\left(2\epsilon_2(K)-3\log3+[K:\Q]\right)\delta^{-2}\left(h(X)+(3\delta+2\delta+\delta\log(n+1))\right)\\
  &\leqslant&(n+6)\exp\left(2\epsilon_2(K)-3\log3+[K:\Q]\right)\max\{\delta^{-2}h(X),\delta^{-1}\}.
  \end{eqnarray*}
We denote by $G_K(X)=H_K(X)^{\frac{n!}{(n-1)(2n)^n}\delta^{-1-\frac{1}{n-1}}}\geqslant1$ for simplicity, where the last inequality is obtained by definition directly. Then by an elementary calculation, we have
\begin{eqnarray*}
  & &\frac{b'(\mathscr X)}{H_K(X)^{\frac{n!}{(n-1)(2n)^n}\delta^{-1-\frac{1}{n-1}}}}\\
  &\leqslant&(n+6)\exp\left(2\epsilon_2(K)-3\log3+[K:\Q]\right)\frac{\max\left\{\frac{\delta^{-1+\frac{1}{n-1}}(n-1)(2n)^n}{n![K:\Q]}\log G_K(X),\delta^{-1}\right\}}{G_K(X)}\\
  &\leqslant&\frac{(n+6)(n-1)(2n)^n\exp\left(2\epsilon_2(K)-3\log3+[K:\Q]\right)}{n!}\delta^{-1+\frac{1}{n-1}}.
\end{eqnarray*}
We have the assertion by taking the above estimate into Theorem \ref{upper bound of degree of general case}.
  \end{proof}

  Compared with Corollary \ref{uniform upper bound of the degree of auxiliary hypersurface}, the result below has a better dependence on the degree of the original hypersurface but a little worse dependence on the bound of heights.
  \begin{coro}\label{uniform upper bound of the degree of auxiliary hypersurface2}
        We keep all the notations and conditions in Theorem \ref{upper bound of degree of general case}. Then there exists a hypersurface of degree smaller than \[\frac{C'_3(n,K)\delta^{3-\frac{1}{n-1}}}{H_K(X)^{\frac{n!}{(n-1)(2n)^n}\delta^{-1-\frac{1}{n-1}}}}B^{{n}/{\left((n-1)\delta^{1/(n-1)}\right)}}\max\left\{\frac{\log B}{[K:\Q]},1\right\},\] which covers $S(X;B)$ but does not contain the generic point of $X$, where the constant
\[C'_3(n,K)=-e^{C_2(n,K)}\frac{C_1(n)(n+6)(2n)^{n}\exp\left(2\epsilon_2(K)-3\log3+[K:\Q]\right)}{(n-1)!},\]
$C_1(n)$ is defined in \eqref{constant C_1}, $C_2(n,K)$ is defined in Theorem \ref{upper bound of degree of general case}, and $H_K(X)$ is defined in Definition \ref{classic height of hypersurface}.
  \end{coro}
\begin{proof}
 If
    \[\frac{\log B}{[K:\Q]}<\frac{(n-1)!}{\delta(2n)^n}h(X)+C_1(n),\]
then by Proposition \ref{naive siegal lemma}, $S(X;B)$ can be covered by a hypersurface of degree no more than $2\delta+n-1$ which does not contain the generic point of $X$. The upper bound of the degree satisfies the bound provided in the statement.

If
\[\frac{\log B}{[K:\Q]}\geqslant\frac{(n-1)!}{\delta(2n)^n}h(X)+C_1(n),\]
which is equivalent to
\[h(X)\leqslant \frac{\delta(2n)^n}{(n-1)!}\cdot\frac{\log B}{[K:\Q]}-\frac{\delta(2n)^n}{(n-1)!}C_1(n),\]
then we deal with it as following. Same as the proof of Corollary \ref{uniform upper bound of the degree of auxiliary hypersurface}, we have
  \[b'(\mathscr X)\leqslant(n+6)\exp\left(2\epsilon_2(K)-3\log3+[K:\Q]\right)\max\{\delta^{-2}h(X),\delta^{-1}\},\]
  where $b'(\mathscr X)$ is the same as that in Proposition \ref{estimate of log p/p for non-geometrically integral 2} and Theorem \ref{upper bound of degree of general case}. Then we have
  \begin{eqnarray*}
& &\frac{b'(\mathscr X)}{H_K(X)^{\frac{n!}{(n-1)(2n)^n}\delta^{-1-\frac{1}{n-1}}}}\\
&\leqslant&\frac{(n+6)\exp\left(2\epsilon_2(K)-3\log3+[K:\Q]\right)\delta^{-1}}{H_K(X)^{\frac{n!}{(n-1)(2n)^n}\delta^{-1-\frac{1}{n-1}}}}\max\left\{ \frac{(2n)^n}{(n-1)!}\left(\frac{\log B}{[K:\Q]}-C_1(n)\right),1\right\}\\
&\leqslant&\frac{-C_1(n)(n+6)(2n)^{n}\exp\left(2\epsilon_2(K)-3\log3+[K:\Q]\right)}{(n-1)!H_K(X)^{\frac{n!}{(n-1)(2n)^n}\delta^{-1-\frac{1}{n-1}}}}\delta^{-1}\max\left\{\frac{\log B}{[K:\Q]},1\right\},
  \end{eqnarray*}
  and we obtain the assertion by taking the above inequality into Theorem \ref{upper bound of degree of general case}.
\end{proof}
 \section{Counting rational points in plane curves}
 As applications of Corollary \ref{uniform upper bound of the degree of auxiliary hypersurface} and Corollary \ref{uniform upper bound of the degree of auxiliary hypersurface2}, we have the following uniform upper bounds of the number of rational points with bounded height in plane curves.
 \subsection{A generalization over an arbitrary number field}
The following result generalizes \cite[Theorem 2]{Cluckers2019} over an arbitrary number field, which has the optimal dependence on the bound of heights.
 \begin{theo}\label{upper bound of rational points of bounded height in plane curves}
   Let $X$ be a geometrically integral curves in $\mathbb P^2_K$ of degree $\delta$. Then we have
   \[\#S(X;B)\leqslant C_3(2,K)\delta^4B^{{2}/{\delta}},\]
   where the constant $C_3(2,K)$ is defined in Corollary \ref{uniform upper bound of the degree of auxiliary hypersurface}. In addition, we have
   \[\#S(X;B)\ll_K\delta^4B^{{2}/{\delta}}.\]
 \end{theo}
 \begin{proof}
We apply the B\'ezout Theorem in the intersection theory (cf. \cite[Proposition 8.4]{Fulton}) to $X$ and the auxiliary hypersurface determined in Corollary \ref{uniform upper bound of the degree of auxiliary hypersurface} for the case of $n=2$, and then we obtain the result.
 \end{proof}

\subsection{A better dependence on the degree}
In this part, we will provide another uniform upper bound of rational points with bounded height in plane curves. This result has a better dependence on the degree than that of Theorem \ref{upper bound of rational points of bounded height in plane curves}, but a bit worse dependence on the bound of heights.
\begin{theo}\label{upper bound of rational points of bounded height in plane curves2}
   Let $X$ be a geometrically integral curves in $\mathbb P^2_K$ of degree $\delta$. Then we have
   \[\#S(X;B)\leqslant C'_3(2,K)\delta^3B^{{2}/{\delta}}\max\left\{\frac{\log B}{[K:\Q]},1\right\},\]
   where the constant $C'_3(2,K)$ is defined in Corollary \ref{uniform upper bound of the degree of auxiliary hypersurface2}. In addition, we have
   \[\#S(X;B)\ll_K\delta^3B^{{2}/{\delta}}\log B\]
   when $B\geqslant\exp([K:\Q])$.
 \end{theo}
 \begin{proof}
This is the same application of B\'ezout Theorem in the intersection theory (cf. \cite[Proposition 8.4]{Fulton}) to Corollary \ref{uniform upper bound of the degree of auxiliary hypersurface2} as that of Corollary \ref{upper bound of rational points of bounded height in plane curves} when $n=2$, where we take $H_K(X)\geqslant1$ defined in Definition \ref{classic height of hypersurface} into consideration.
 \end{proof}
 \begin{rema}\label{conj of HB}
   It seems that the upper bound given in Theorem \ref{upper bound of rational points of bounded height in plane curves} and Theorem \ref{upper bound of rational points of bounded height in plane curves2} are not optimal. Actually, for a geometrically integral plane curve $X\hookrightarrow\P^2_{\Q}$ of degree $\delta$, Heath-Brown conjectured the uniform upper bound
   \[\#S(X;B)\ll\delta^2B^{{2}/{\delta}}.\]
   By the examples given in \cite[\S6]{Cluckers2019}, the exponent $2$ of $\delta$ in the above conjecture would be optimal.
 \end{rema}
\section{Explicit estimates under the assumption of GRH}
Under the assumption of \textbf{GRH (the Generalized Riemann Hypothesis)} for the Dedekind zeta function of the number field $K$, we have more explicit estimates of
\[\theta_K(x)=\sum\limits_{N(\p)\leqslant x}\log N(\p),\;\psi_K(x)=\sum\limits_{N(\p)\leqslant x}\frac{\log N(\p)}{N(\p)} \hbox{, and }\phi_K(x)=\sum\limits_{N(\p)\leqslant x}\frac{\log N(\p)}{N(\p)^\frac{3}{2}},\]
where $x\in\mathbb R^+$, $\p\in\spm\O_K$, and $N(\p)=\#(\O_K/\p)$. By these results, we are able to obtain more explicit estimates of in the global  determinant method and more explicit estimates in the densities of rational points. In this section, we will give explicit estimates of the remainders of the above $\theta_K(x)$, $\psi_K(x)$ and $\phi_K(x)$.

\subsection{Explicit estimates of the distribution of prime ideals with bounded norms}
In order to obtain explicit estimates, first we refer to a result in \cite{grenie2016explicit} under the assumption of GRH. Then by the same technique in \cite{Rosen1999}, we get an explicit generalization of Mertens' first theorem. If we do not need the explicit version, it is not necessary to assume GRH.

Let $\Delta_K$ be the discriminant of the number field $K$. By \cite[Corollary 1.3]{grenie2016explicit}, if $x\geqslant3$, we have
\begin{eqnarray*}
  \left|\theta_K(x)-x\right|&\leqslant&\sqrt{x}\Bigg(\left(\frac{1}{2\pi}\log\left(\frac{18.8x}{\log^2x}\right)+2.3\right)\log\Delta_K\\
  & &\;+\left(\frac{1}{8\pi}\log^2\left(\frac{18.8x}{\log^2x}\right)+1.3\right)[K:\Q]+0.3\log x+14.6\Bigg)\\
  &\leqslant&\sqrt{x}\log^2x\left(\log\Delta_K+[K:\Q]\right)\Bigg(\frac{1}{2\pi(\log^2x)}\log\left(\frac{18.8x}{\log^2x}\right)+3.6\\
  & &\;+\frac{1}{8\pi(\log^2x)}\log^2\left(\frac{18.8x}{\log^2x}\right)+0.3\frac{1}{\log x}+14.6\frac{1}{\log^2x}\Bigg)
\end{eqnarray*}
under the assumption of GRH. Since $x\geqslant3$, we obtain
\begin{equation}\label{explicit estimate of theta_K}
  \left|\theta_K(x)-x\right|\leqslant528\sqrt{x}\log^2x\left(\log\Delta_K+[K:\Q]\right)
  \end{equation}
by Mathematica under the same assumption.

Same as the application of \cite[Lemma 2.1]{Rosen1999} to the proof of \cite[Lemma 2.3]{Rosen1999}, by Abel's summation formula, we have
\[\psi_K(x)=\frac{\theta_K(x)}{x}+\int_2^x\frac{\theta_K(t)}{t^2}dt.\]
From \eqref{explicit estimate of theta_K}, we have
\[\left|\psi_K(x)-1-\int_2^x\frac{1}{t}dt\right|\leqslant528\left(\log\Delta_K+[K:\Q]\right)\left(\frac{\log^2x}{\sqrt{x}}+\int_2^x\frac{\log^2t}{t^{\frac{3}{2}}}dt\right).\]
Then by an elementary calculation executed by Mathematica, we obtain
\begin{equation}\label{explicit estimate of psi_K}
  \left|\psi_K(x)-\log x\right|\leqslant9550\left(\log\Delta_K+[K:\Q]\right)
\end{equation}
for $x\geqslant3$ under the assumption of GRH.

The same goes for the above, we have
\[\phi_K(x)=\frac{\theta_K(x)}{x^\frac{3}{2}}+\frac{3}{2}\int_2^x\frac{\theta_K(t)}{t^\frac{5}{2}}dt.\]
Then from \eqref{explicit estimate of theta_K}, we obtain
\[\left|\phi_K(x)-\frac{1}{\sqrt{x}}-\frac{3}{2}\int_2^x\frac{1}{t^{\frac32}}dt\right|\leqslant792\left(\log\Delta_K+[K:\Q]\right)\left(\frac{\log^2x}{x}+\int_2^x\frac{\log^2t}{t^2}dt\right).\]
Then after an elementary calculation executed by Mathematica, we have
\begin{equation}\label{explicit estmiate of phi_K}
  \left|\phi_K(x)-\frac{3}{2}\sqrt{2}+\frac{2}{x^\frac{1}{2}}\right|\leqslant2516\frac{\log^2x}{x}\left(\log\Delta_K+[K:\Q]\right)
\end{equation}
uniformly for all $x\in\mathbb R^+$ under the assumption of GRH.

\subsection{Estimates of remainders}
We compare \eqref{implicit estimate of theta_K} with \eqref{explicit estimate of theta_K}, \eqref{implicit estimate of psi_K} with \eqref{explicit estimate of psi_K}, and \eqref{implicit estimate of phi_K} with \eqref{explicit estmiate of phi_K}. Then we can suppose
\[\epsilon_1(K,x)=528\sqrt{x}\log^2x\left(\log\Delta_K+[K:\Q]\right)\]
in \eqref{implicit estimate of theta_K},
\[\epsilon_2(K)=9550\left(\log\Delta_K+[K:\Q]\right)\]
in \eqref{implicit estimate of psi_K}, and
\[\epsilon_3(K,x)=2516\frac{\log^2x}{x}\left(\log\Delta_K+[K:\Q]\right)\]
in \eqref{implicit estimate of phi_K} under the assumption of GRH. We take the above estimates of $\epsilon_1(K,x)$, $\epsilon_2(K)$ and $\epsilon_3(K,x)$ into Proposition \ref{estimate of log p/p for non-geometrically integral 2}, \eqref{kappa_1(K)} and \eqref{kappa_2(K)}, then we obtain a more explicit estimate in Theorem \ref{upper bound of degree of general case}. Hence the estimates in Corollary \ref{uniform upper bound of the degree of auxiliary hypersurface}, Corollary \ref{uniform upper bound of the degree of auxiliary hypersurface2}, Theorem \ref{upper bound of rational points of bounded height in plane curves} and Theorem \ref{upper bound of rational points of bounded height in plane curves2} are more explicit under the assumption of GRH.
\begin{rema}
  Besides the case of $K=\Q$, if we work over some other particular number fields, the assumption of the Generalized Riemann Hypothesis may not be obligatory. For example, in \cite[Theorem 2]{grzeskowiak2017explicit}, we are able to do it over totally imaginary fields. It depends on the understanding of the zero-free region of the Dedekind zeta function of the number field $K$.
\end{rema}
\appendix
\section{An explicit lower bound of $Q_{\xi}(r)$}\label{Section: estimate of Q}
In this appendix, we give an explicit lower bound of the function $Q_\xi(r)$ defined in \eqref{Q_r} for the case of hypersurfaces.

 In the following proof of Proposition \ref{estimate of Q}, the inequality
 \[\frac{(N-m+1)^m}{m!}\leqslant{N \choose m}\leqslant \frac{\left(N-(m-1)/2\right)^m}{m!}\]
 will be very useful, where $N$ and $m$ are two positive integers, and $N\geqslant m\geqslant1$.
\begin{prop}\label{estimate of Q}
Let $X$ be a hypersurface of $\mathbb P^n_k$, $\xi$ be a closed point in $X$, and $\mu_\xi$ be the multiplicity of $\xi$ in $X$ induced by its local Hilbert-Samuel function. The function $Q_\xi(r)$ is defined in the equality \eqref{Q_r}. Then we have
  \begin{eqnarray*}
    Q_\xi(r)&>&\left(\frac{(n-1)!}{\mu_\xi}\right)^{\frac{1}{n-1}}\left(\frac{n-1}{n}\right)r^{\frac{n}{n-1}}-\frac{n^3+2n^2+n-4}{2n(n+1)}r.
  \end{eqnarray*}
\end{prop}
\begin{proof}
For the case of hypersurfaces, by \cite[Example 2.70 (2)]{Kollar2007}, we have
\begin{equation*}
  H_\xi(s)={n+s-1\choose s}-{n+s-\mu_\xi-1\choose s-\mu_\xi}.
\end{equation*}

 We define the function $U_\xi(k)=H_\xi(0)+\cdots+H_\xi(k)$, then we have
  \begin{eqnarray*}
    U_\xi(k)&=&\sum\limits_{j=0}^k{n+j-1\choose j}-\sum\limits_{j=0}^k{n+j-\mu_\xi-1\choose j-\mu_\xi}\\
    &=& {n+k\choose n}-{n+k-\mu_\xi\choose n}.
  \end{eqnarray*}
Then we obtain
  \begin{eqnarray*}
    Q_\xi(U_\xi(k))&=&\sum\limits_{j=0}^njH_\xi(j)\\
    &=&\sum\limits_{j=0}^kj{j+n-1\choose n-1}-\sum\limits_{j=0}^kj{n+j-\mu_\xi-1\choose n-1}\\
    &=& n{k+n\choose n+1}-n{k-\mu_\xi+n\choose n+1}-\mu_\xi{n+k-\mu_\xi\choose n}.
  \end{eqnarray*}

Let $r\in]U_\xi(k-1),U_\xi(k)]$. By the definition of $Q_\xi(r)$ in the equality \eqref{Q_r}, we have the inequality
\[Q_\xi(U_\xi(k-1))\leqslant Q_\xi(r)\leqslant Q_\xi(U_\xi(k)).\]
So we have
\begin{eqnarray}\label{Q(r)-xi}
  Q_\xi(r)&=&Q_\xi(U_\xi(k-1))+k(r-U_\xi(k-1))\nonumber\\
  &=&n{n+k-1\choose n+1}-n{n+k-\mu_\xi-1\choose n+1}-\mu_\xi{n+k-\mu_\xi-1\choose n}\nonumber\\
  & &+kr-k{n+k-1\choose n}+k{n+k-\mu_\xi-1\choose n}\nonumber\\
  &=&kr+{n+k-\mu_\xi\choose n+1}-{n+k\choose n+1}.
\end{eqnarray}

In order to get a lower bound of $Q_\xi(r)$, we need to estimate the term
\[{n+k-\mu_\xi\choose n+1}-{n+k\choose n+1}.\]
In fact, we have the estimate
\begin{eqnarray*}
  & &\left[{n+k\choose n+1}-{n+k-\mu_\xi\choose n+1}\right]\Big/U_\xi(k-1)\\
&=&\left[{n+k\choose n+1}-{n+k-\mu_\xi\choose n+1}\right]\Big/\left[{n+k-1\choose n}-{n+k-\mu_\xi-1\choose n}\right]\\
  &=&\frac{(n+k)(n+k-1)\cdots k-(n+k-\mu_\xi)\cdots(k-\mu_\xi)}{(n+1)[(n+k-1)\cdots k-(n+k-\mu_\xi-1)\cdots(k-\mu_\xi)]}\\
  &=&\left(\frac{(n+k)(n+k-1)}{k-1}\frac{(n+k-2)\cdots(k-1)}{(n+k-\mu_\xi-1)\cdots(k-\mu_\xi)}-(n+k-\mu_\xi)\right)\Big/\\
& &\left(\left[\frac{n+k-1}{k-1}
  \left[\frac{(n+k-2)\cdots(k-1)}{(n+k-\mu_\xi-1)\cdots(k-\mu_\xi)}\right]-1\right](n+1)\right)\\
  &=&\frac{1}{n+1}\left[n+k+\frac{\mu_\xi}{\frac{n+k-1}{k-1}\cdot\frac{(n+k-2)\cdots(k-1)}{(n+k-\mu_\xi-1)\cdots(k-\mu_\xi)}-1}\right]\\
  &\leqslant&\frac{1}{n+1}\left[n+k+\frac{\mu_\xi}{\left(\frac{n+k-1}{n+k-\mu_\xi-1}\right)^n-1}\right]\\
  &=&\frac{1}{n+1}\left[n+k+\frac{(n+k-\mu_\xi-1)^n}{(n+k-1)^{n-1}+\cdots+(n+k-\mu_\xi-1)^{n-1}}\right]\\
  &\leqslant&\frac{1}{n+1}\left(n+k+\frac{n+k-\mu_\xi-1}{n}\right)\\
  &=&\frac{(n+1)k+n^2+n-\mu_\xi-1}{n(n+1)}.
\end{eqnarray*}

By the equality \eqref{Q(r)-xi}, we obtain
\begin{eqnarray}\label{U(k-1)<r}
  Q_\xi(r)&=&kr+{n+k-\mu_\xi\choose n+1}-{n+k\choose n+1}\nonumber\\
  &\geqslant&kr-\frac{(n+1)k+n^2+n-\mu_\xi-1}{n(n+1)}U_\xi(k-1)\\
  &>&kr-\frac{(n+1)k+n^2+n-\mu_\xi-1}{n(n+1)}r\nonumber\\
  &=&\frac{(n^2-1)k-n^2-n+\mu_\xi+1}{n(n+1)}r\nonumber\\
  &=&\left(\frac{n-1}{n}k-\frac{n^2+n-\mu_\xi-1}{n(n+1)}\right)r\nonumber,
\end{eqnarray}
where we use the estimate $U_\xi(k-1)<r$ in the inequality \eqref{U(k-1)<r}. In addition, we obtain the inequality
\begin{eqnarray*}
  r\leqslant U_\xi(k)&=&{n+k\choose n}-{n+k-\mu_\xi\choose n}=\sum\limits_{j=1}^{\mu_\xi}{n+k-j\choose n-1}\\&\leqslant&\frac{1}{(n-1)!}\sum\limits_{j=1}^{\mu_\xi}(k+\frac{n}{2}-j+1)^{n-1}.
\end{eqnarray*}

In addition, we have
\begin{equation*}
  \sum\limits_{j=1}^{\mu_\xi}(k+\frac{n}{2}-j+1)^{n-1}\leqslant\mu_\xi\left(k+\frac{n}{2}\right)^{n-1}.
\end{equation*}
Then
\begin{equation*}
  k\geqslant\frac{1}{\sqrt[n-1]{\mu_\xi}}\sqrt[n-1]{(n-1)!r}-\frac{n}{2}.
\end{equation*}

Finally we have
\begin{eqnarray*}
  Q_\xi(r)&>&\left(\frac{n-1}{n}\left(\frac{1}{\sqrt[n-1]{\mu_\xi}}\sqrt[n-1]{(n-1)!r}-\frac{n}{2}\right)-\frac{n^2+n-\mu_\xi-1}{n(n+1)}\right)r\\
  &\geqslant&\left(\frac{(n-1)!}{\mu_\xi}\right)^{\frac{1}{n-1}}\left(\frac{n-1}{n}\right)r^{\frac{n}{n-1}}-\frac{n^3+2n^2+n-4}{2n(n+1)}r,
\end{eqnarray*}
for $\mu_\xi\geqslant1$. Then we obtain the result.
\end{proof}
\backmatter

\bibliography{liu}
\bibliographystyle{smfplain}

\end{document}